\documentclass[british,magyar,english]{article}
\usepackage[T1]{fontenc}
\usepackage[latin2,latin9]{inputenc}
\usepackage{geometry}
\usepackage{units}
\usepackage{amsmath}
\usepackage{amsthm}
\usepackage{amssymb}
\usepackage{graphicx}

\makeatletter
\theoremstyle{plain}
\newtheorem{thm}{\protect\theoremname}[section]
  \theoremstyle{definition}
  \newtheorem{defn}[thm]{\protect\definitionname}
  \theoremstyle{remark}
  \newtheorem{rem}[thm]{\protect\remarkname}
  \theoremstyle{plain}
  \newtheorem{prop}[thm]{\protect\propositionname}
  \theoremstyle{plain}
  \newtheorem{cor}[thm]{\protect\corollaryname}
  \theoremstyle{plain}
  \newtheorem{lem}[thm]{\protect\lemmaname}
  \theoremstyle{definition}
  \newtheorem{example}[thm]{\protect\examplename}

\makeatother

\usepackage{babel}
  \addto\captionsbritish{\renewcommand{\corollaryname}{Corollary}}
  \addto\captionsbritish{\renewcommand{\definitionname}{Definition}}
  \addto\captionsbritish{\renewcommand{\examplename}{Example}}
  \addto\captionsbritish{\renewcommand{\lemmaname}{Lemma}}
  \addto\captionsbritish{\renewcommand{\propositionname}{Proposition}}
  \addto\captionsbritish{\renewcommand{\remarkname}{Remark}}
  \addto\captionsbritish{\renewcommand{\theoremname}{Theorem}}
  \addto\captionsenglish{\renewcommand{\corollaryname}{Corollary}}
  \addto\captionsenglish{\renewcommand{\definitionname}{Definition}}
  \addto\captionsenglish{\renewcommand{\examplename}{Example}}
  \addto\captionsenglish{\renewcommand{\lemmaname}{Lemma}}
  \addto\captionsenglish{\renewcommand{\propositionname}{Proposition}}
  \addto\captionsenglish{\renewcommand{\remarkname}{Remark}}
  \addto\captionsenglish{\renewcommand{\theoremname}{Theorem}}
  \addto\captionsmagyar{\renewcommand{\corollaryname}{Következmény}}
  \addto\captionsmagyar{\renewcommand{\definitionname}{Definíció}}
  \addto\captionsmagyar{\renewcommand{\examplename}{Példa}}
  \addto\captionsmagyar{\renewcommand{\lemmaname}{Segédtétel}}
  \addto\captionsmagyar{\renewcommand{\propositionname}{Állítás}}
  \addto\captionsmagyar{\renewcommand{\remarkname}{Észrevétel}}
  \addto\captionsmagyar{\renewcommand{\theoremname}{Tétel}}
  \providecommand{\corollaryname}{Corollary}
  \providecommand{\definitionname}{Definition}
  \providecommand{\examplename}{Example}
  \providecommand{\lemmaname}{Lemma}
  \providecommand{\propositionname}{Proposition}
  \providecommand{\remarkname}{Remark}
\providecommand{\theoremname}{Theorem}

\begin{document}
\date{}

\title{Interval projections of self-similar sets}

\author{\'Abel Farkas}
\maketitle
\begin{abstract}
We show that if $K$ is a self-similar $1$-set that is not contained
in a line and either satisfies the strong separation condition or
is defined via homotheties then there are at most finitely many lines
through the origin such that the projection of $K$ onto them is an
interval.
\end{abstract}

\section{Introduction}

\subsection{Overview}

The measure theory of projections of `fractals' has gained much attention
in the past few decades. A seminal result is that for a Borel set
$K\subseteq\mathbb{R}^{2}$ if $\dim_{H}(K)>1$ then
\begin{equation}
\mathcal{H}^{1}(\Pi_{M}(K))>0\label{eq:marstrand pos measuree}
\end{equation}
for almost all lines $M$, where $\dim_{H}$ denotes the Hausdorff
dimension, $\mathcal{H}^{s}$ denotes the $s$-dimensional Hausdorff
measure and $\Pi_{M}:\mathbb{R}^{2}\longrightarrow M$ denotes the
orthogonal projection onto $M$. This was proved by Marstrand \cite{Marstrand_paper},
and was generalized to higher dimensions by Mattila \cite{Mattila-H-dim of projections of sets}.
When $\dim_{H}(K)<1$ then $\mathcal{H}^{1}(\Pi_{M}(K))=0$ for every
line $M$ since projection does not increase the Hausdorff dimension.
In the critical case when $\dim_{H}(K)=1$ two things can happen.
A set $K\subseteq\mathbb{R}^{d}$ is called an \textit{$s$-set} if
$0<\mathcal{H}^{s}(K)<\infty$. We call a $1$-set $K$ \textsl{purely
$1$-unrectifiable} if $\mathcal{H}^{1}(K\cap M)=0$ for every differentiable
$1$-manifold $M$. It was shown by Besicovitch \cite{Besicovitch-Proj THM}
and generalised to higher dimensions by Federer \cite{Federer-Proj THM}
that for a $1$-set $K\subseteq\mathbb{R}^{2}$
\[
\mathcal{H}^{1}(\Pi_{M}(K))=0
\]
for almost all lines $M$ if and only if $K$ is purely $1$-unrectifiable.
If $K$ is not purely $1$-unrectifiable then $\mathcal{H}^{1}(\Pi_{M}(K))>0$
for all but at most one lines $M$.

When $K\subseteq\mathbb{R}$ is a dynamically defined set often the
case is that $\mathcal{H}^{1}(K)>0$ if and only if $K$ contains
an interval. For example, if $K\subseteq\mathbb{R}$ is a self-similar
set satisfying the `open set condition' then $\mathcal{H}^{1}(K)>0$
if and only if $K$ contains an interval \cite[Corollary 2.3]{Schief OSC}.
Whether the statement still holds without assuming the open set condition
is an intriguing, still open question. If $K\subseteq\mathbb{R}^{2}$
is a self-similar $1$-set satisfying the open set condition then
for a line $M$ we have that $\mathcal{H}^{1}(\Pi_{M}(K)>0$ if and
only if $\Pi_{M}(K)$ contains an interval (see \cite[Thm 1.1; Thm 1.5]{linim},
\cite{Wang GDA OSC} and \cite[Corollary 2.3]{Schief OSC}).

In this paper we are concerned about the situation when the projection
$\Pi_{M}(K)$ not only contains an interval but is an interval itself.
The unit semicircle in the plane is a $1$-set and the projection
of it onto every line is an interval. Falconer and Fraser \cite{Falconer-Fraser-visible part}
studied the visible part of self-similar sets that project onto an
interval in every direction. Our main results show that under some
natural assumptions on a self-similar $1$-set $K$ there are at most
finitely many lines $M$ such that $\Pi_{M}(K)$ is an interval in
contrast with the example of the unit semicircle which is a self-conformal
set. This is the case when $K$ satisfies the `strong separation condition'.
We also show that we can drop the assumption of the strong separation
condition if we assume that every defining map is a homothety. In
the last section we give examples of self-similar sets with several
interval projections. Finally, we establish an invariance property
of the moment of inertia of a self-similar $1$-set with several interval
projections.

\subsection{Definitions and notations}

For integers $0\leq l<d$ and for an $l$-dimensional affine subspace
$M\subseteq\mathbb{R}^{d}$ we denote the orthogonal projection onto
$M$ by $\Pi_{M}:\mathbb{R}^{d}\longrightarrow M$. Throughout the
paper we consider $\Pi_{M}$ as a mapping into $\mathbb{R}^{l}=M$.

\begin{defn}
\label{Def: Int-Proj}Let $M\subseteq\mathbb{R}^{2}$ be a line through
the origin and $K\subseteq\mathbb{R}^{2}$ be an arbitrary set. We
call the projection $\Pi_{M}:\mathbb{R}^{2}\longrightarrow M$ an
\textit{interval projection of $K$} if $\Pi_{M}(K)$ is an interval
with the $M=\mathbb{R}$ identification. For a single point $x\in\mathbb{R}$
we consider $\{x\}$ to be the closed interval $[x,x]$ of length
$0$.
\end{defn}

A \textit{self-similar iterated function system} (SS-IFS) in $\mathbb{R}^{d}$
is a finite collection of maps $\left\{ S_{i}\right\} _{i=1}^{m}$
from $\mathbb{R}^{d}$ to $\mathbb{R}^{d}$ such that all the $S_{i}$
are contracting similarities. The \textit{attractor} of the SS-IFS
is the unique nonempty compact set $K$ such that $K=\bigcup_{i=1}^{m}S_{i}(K)$.
The attractor of an SS-IFS is called a \textit{self-similar set}.
We say that the SS-IFS $\left\{ S_{i}\right\} _{i=1}^{m}$ satisfies
the \textit{strong separation condition} (SSC) if the $\left\{ S_{i}(K)\right\} _{i=1}^{m}$
are a disjoint. Every $S_{i}$ can be uniquely decomposed as
\begin{equation}
S_{i}(x)=r_{i}T_{i}(x)+v_{i}\label{eq:S_ideconposition}
\end{equation}
for all $x\in\mathbb{R}^{d}$, where $0<r_{i}<1$, $T_{i}$ is an
orthogonal transformation and $v_{i}\in\mathbb{R}^{d}$ is a translation
vector. Let $\mathcal{T}$ denote the group generated by the orthogonal
transformations $\left\{ T_{i}\right\} _{i=1}^{m}$. We denote the
set $\left\{ 1,2,\ldots,m\right\} $ by $\mathcal{I}$. Let $\boldsymbol{\mathbf{i}}=(i_{1},\ldots,i_{k})\in\mathcal{I}^{k}$
i.e. a $k$-tuple of indices. Then we write $S_{\boldsymbol{\mathbf{i}}}=S_{i_{1}}\circ\ldots\circ S_{i_{k}}$
and $K_{\boldsymbol{\mathbf{i}}}=S_{\boldsymbol{\mathbf{i}}}(K)$.
Since the similarities are decomposed as in (\ref{eq:S_ideconposition})
we write $r_{\boldsymbol{\mathbf{i}}}=r_{i_{1}}\cdot\ldots\cdot r_{i_{k}}$
and $T_{\boldsymbol{\mathbf{i}}}=T_{i_{1}}\circ\ldots\circ T_{i_{k}}$.

\subsection{Finiteness of interval projections}

In this section we list the main results of this paper. The proofs
are provided in Section \ref{sec:Isolated-interval-projections} and
Section \ref{sec:Interval-projections-of}.

We provide examples, Example \ref{Ex: Minen-vet-int}, of self-similar
sets with Hausdorff dimension arbitrarily close to $1$ such that
projection of them onto every line is an interval. However, when the
Hausdorff dimension is $1$ we prove that there are only finitely
many such lines. We also provide an example of a totally disconnected,
compact, non-self-similar set of Hausdorff dimension $1$ that projects
onto an interval in every direction (see Example \ref{1dimfullip}).
\begin{thm}
\label{thm:SS int-proj SSC}Let $\left\{ S_{i}\right\} _{i=1}^{m}$
be a self-similar iterated function system in $\mathbb{R}^{2}$ with
attractor $K$ such that $\left\{ S_{i}\right\} _{i=1}^{m}$ satisfies
the strong separation condition and $K$ is a $1$-set. Then there
are at most finitely many lines through the origin such that the orthogonal
projection onto them are interval projections of $K$.
\end{thm}

The $1$-dimensional Sierpinski triangle $K_{\bigtriangleup}$ (see
Example \ref{exa:1dim-sierp}) is a self-similar $1$-set and the
usual SS-IFS for $K_{\bigtriangleup}$ satisfies the SSC and $\left|\mathcal{T}\right|=1$,
where $\left|.\right|$ denotes the cardinality of a set. Hence the
projection onto every line is a self-similar set with `similarity
dimension' $1$ and so by \cite[Corollary (2.3)]{Schief OSC} the
projection has positive $\mathcal{H}^{1}$-measure if and only if
the projection contains an interval. Kenyon \cite{Kenyon-sierpinski}
showed that this occurs exactly for a countable and dense set of lines
through the origin. So there exists a dense set of lines through the
origin onto which the projection contains an interval, but the projection
onto at most finitely many of them is an interval projection by Theorem
\ref{thm:SS int-proj SSC}. One can show that $K_{\bigtriangleup}$
has exactly three interval projections. In Example \ref{exa:negyzet-ip}
and Example \ref{exa:four corner} we provide examples of self-similar
$1$-sets such that the projection of them onto four different lines
are intervals.

While Theorem \ref{thm:SS int-proj SSC} requires the SSC, we would
like to eliminate this separation condition but to do so we need a
further assumption on $\mathcal{T}$. A similarity $S:\mathbb{R}^{d}\longrightarrow\mathbb{R}^{d}$
is called a \textit{homothety} if there are $r\in\mathbb{R}\setminus\{0\}$
and $v\in\mathbb{R}^{2}$ such that $S(x)=rx+v$ for every $x\in\mathbb{R}^{d}$.
\begin{thm}
\label{thm:SS int-proj OSC}Let $\left\{ S_{i}\right\} _{i=1}^{m}$
be a self-similar iterated function system in $\mathbb{R}^{2}$ with
attractor $K$ such that all $S_{i}$ are homotheties and $K$ is
a $1$-set that is not contained in any line. Then there are at most
finitely many lines through the origin such that the orthogonal projection
onto them are interval projections of $K$.
\end{thm}

\begin{rem}
\label{rem:invar finiteness}The reason why we need in the proof that
every $S_{i}$ is a homothety is that the set of interval projections
$IP(K)$ is invariant under the action of $\mathcal{T}$, i.e. if
$\Pi_{M}\in IP(K)$ then $\Pi_{T(M)}\in IP(K)$ for every $T\in\mathcal{T}$
. Hence via a similar argument to the one in the proof of Theorem
\ref{thm:SS int-proj OSC} one can show that instead of assuming that
every similarity is a homothety it is enough to assume that $IP(K)$
is invariant under the action of $\mathcal{T}$.
\end{rem}

In Theorem \ref{thm:SS int-proj OSC} the assumption, that $K$ is
not contained in any line, is essential because a non-degenerate line
segment $[x,y]$ on the plane is a $1$-set and is the attractor of
some SS-IFS that contains only homotheties, but the orthogonal projection
onto every line is an interval projection of $[x,y]$. Due to Remark
\ref{rem:invar finiteness} the only self similar sets in the plane
of Hausdorff dimension $1$ that project onto an interval in every
direction are line segments.

In Proposition \ref{prop: Tangent plane prop} and in Corollary \ref{cor: manifold intersection cor}
we generalize a result of Mattila \cite[Propoition 4.2, Corollary 4.3]{Mattila-rectifiablility}
on the unrectifiability of self-similar sets as we get rid of the
`open set condition' from the assumptions. This generalization will
play an important role in the proof of Theorem \ref{thm:SS int-proj OSC}.

The following result of Farkas \cite[Theorem 1.5]{linim} implies
that if the similarities contain an irrational rotation then $K$
has no interval projection.
\begin{thm}
\label{thm:infinite T no IP}Let $\left\{ S_{i}\right\} _{i=1}^{m}$
be an SS-IFS in $\mathbb{R}^{2}$ with attractor $K$ and assume that
$\mathcal{H}^{1}\left(K\right)<\infty$. If $\left|\mathcal{T}\right|=\infty$
then $\mathcal{H}^{1}(\Pi_{M}(K))=0$ for every line $M$. Hence $K$
has no interval projection.
\end{thm}

\subsection{Moment of inertia}

For $\theta\in\mathbb{R}$ let $L_{\theta}\subseteq\mathbb{R}^{2}$
be the line $\left\{ (t\cos\theta,t\sin\theta)\right\} _{t\in\mathbb{R}}$
and let $P_{\theta}$ be the projection $\Pi_{L_{\theta}}$. In Physics
the moment of inertia of a rigid body in $\mathbb{R}^{3}$ with respect
to a rotational axis determines the torque needed for a desired angular
acceleration about the rotational axis. It is the rotational motion
analog of mass for linear motion and is expressed by $\int r^{2}\mathrm{d}\mu$
where $r$ is the distance of a point from the rotational axis and
$\mu$ is the mass distribution of the rigid body. If the body is
a solid of revolution about the rotational axis then it is enough
to analyse a $2$-dimensional cross section of the body through the
rotational axis. In that case $\mu$ is a mass distribution in $\mathbb{R}^{2}=\left\{ (x,y)\in\mathbb{R}^{2}:x,y\in\mathbb{R}\right\} $
so in the rest of this section $\mu$ is a mass distribution in $\mathbb{R}^{2}$.
If the rotational axis is $L_{\frac{\pi}{2}+\theta}$ then the moment
of inertia can be expressed by $\int r^{2}\mathrm{d}\mu=\int\left|P_{\theta}(x,y)\right|^{2}\mathrm{d}\mu(x,y)$.

The next proposition states that if $\mu$ is a mass distribution
in $\mathbb{R}^{2}$ such that the centre of mass is the origin and
the moment of inertia with respect to three different rotational axis
are the same then they are the same with respect to every rotational
axis. This is well-known in mechanics, but for completeness we include
the proof.
\begin{prop}
\label{prop:general moment of  inertia}Let $\mu$ be a finite Borel
measure in $\mathbb{R}^{2}=\left\{ (x,y)\in\mathbb{R}^{2}:x,y\in\mathbb{R}\right\} $
such that $\int x^{2}+y^{2}\mathrm{d}\mu(x,y)<\infty$ and
\[
\int x\mathrm{d}\mu(x,y)=\int y\mathrm{d}\mu(x,y)=0.
\]
Assume that there exist $c>0$ and three different angles $\theta,\phi,\psi\in[0,\pi)$
such that
\[
c=\int\left|P_{\theta}(x,y)\right|^{2}\mathrm{d}\mu(x,y)=\int\left|P_{\phi}(x,y)\right|^{2}\mathrm{d}\mu(x,y)=\int\left|P_{\psi}(x,y)\right|^{2}\mathrm{d}\mu(x,y).
\]
Then $\int x\cdot y\mathrm{d}\mu(x,y)=0$ and $\int\left|P_{\gamma}(x,y)\right|^{2}\mathrm{d}\mu(x,y)=c$
for every $\gamma\in\mathbb{R}$.
\end{prop}

\begin{proof}
\noindent For every $\gamma\in\mathbb{R}$
\[
\int\left|P_{\gamma}(x,y)\right|^{2}\mathrm{d}\mu(x,y)=\int\left(x\cos\gamma+y\sin\gamma\right)^{2}\mathrm{d}\mu(x,y)
\]
\[
=\cos^{2}\gamma\int x^{2}\mathrm{d}\mu(x,y)+\sin^{2}\gamma\int y^{2}\mathrm{d}\mu(x,y)+2\cos\gamma\sin\gamma\int xy\mathrm{d}\mu(x,y)=
\]
\[
\begin{bmatrix}\cos\gamma & \sin\gamma\end{bmatrix}\begin{bmatrix}\int x^{2}\mathrm{d}\mu & \int xy\mathrm{d}\mu\\
\int xy\mathrm{d}\mu & \int y^{2}\mathrm{d}\mu
\end{bmatrix}\begin{bmatrix}\cos\gamma\\
\sin\gamma
\end{bmatrix}.
\]
We can think of the $2\times2$ symmetric matrix in the middle as
a symmetric bilinear form in the plane. The proposition follows from
the fact that a symmetric bilinear form $\beta$ in the plane is uniquely
determined by the quantities $\beta(v,v)$ for three pairwise independent
vectors $v$.
\end{proof}
\begin{rem}
The symmetric bilinear form, appearing in the previous proof, is known
in mechanics as the moment of inertia tensor.
\end{rem}

For a Borel measure $\mu$ on $\mathbb{R}^{2}$, a subspace $M\subseteq\mathbb{R}^{2}$
and a Borel function $f:\mathbb{R}^{2}\longrightarrow M$ let $f^{*}\mu(.)=\mu(f^{-1}(.))$
be the image measure in $f(\mathbb{R}^{2})$. In the case when $f=P_{\theta}$
we identify $P_{\theta}(\mathbb{R}^{2})=L_{\theta}=\left\{ (t\cos\theta,t\sin\theta)\right\} _{t\in\mathbb{R}}$
with $\mathbb{R}$.
\begin{cor}
\label{cor:minden vet measure=00003D}Let $\mu$ be a finite Borel
measure in $\mathbb{R}^{2}=\left\{ (x,y)\in\mathbb{R}^{2}:x,y\in\mathbb{R}\right\} $
such that $\int x^{2}+y^{2}\mathrm{d}\mu(x,y)<\infty$. Assume that
there exist three angles $\theta,\phi,\psi\in[0,\pi)$ such that
\[
P_{\theta}^{*}\mu=P_{\phi}^{*}\mu=P_{\psi}^{*}\mu.
\]
Then $\int x\cdot y\mathrm{d}\mu(x,y)=0$ and there exists $c>0$
such that $\int\left|P_{\gamma}(x,y)\right|^{2}\mathrm{d}\mu(x,y)=c$
for every $\gamma\in\mathbb{R}$.
\end{cor}

If we knew that $\int x\mathrm{d}\mu(x,y)=\int y\mathrm{d}\mu(x,y)=0$
then Corollary \ref{cor:minden vet measure=00003D} would be an immediate
consequence of Proposition \ref{prop:general moment of  inertia}
since $\int\left|P_{\gamma}(x,y)\right|^{2}\mathrm{d}\mu(x,y)=\int\left|(x,y)\right|^{2}\mathrm{d}P_{\gamma}^{*}\mu(x,y)$
for every $\gamma\in\mathbb{R}$. To see that $\int x\mathrm{d}\mu(x,y)=\int y\mathrm{d}\mu(x,y)=0$
we need to show that the centre of mass is the origin. It follows
by an easy trigonometric calculation using the trigonometric identity
$\frac{\cos\alpha-\cos\beta}{\sin\alpha-\sin\beta}=-\tan\left(\frac{\alpha+\beta}{2}\right)$
and that
\[
\int P_{\theta}(x,y)\mathrm{d}\mu(x,y)=\int P_{\phi}(x,y)\mathrm{d}\mu(x,y)=\int P_{\psi}(x,y)\mathrm{d}\mu(x,y).
\]
In other words, if the projection of the centre of mass onto three
different lines has the same distance from the origin then it has
to be the origin. We leave for the reader to check the details.

The following result of Farkas \cite[Theorem 1.3]{linim} says that
if $\mathcal{H}^{1}\left(P_{\theta}(K)\right)>0$ for a self-similar
$1$-set $K$ then the projection measure is constant times the restriction
of Lebesgue measure to the projection.
\begin{thm}
\label{thm:vet meaasure SS}Let $K\subseteq\mathbb{R}^{2}$ be a self-similar
$1$-set. If $\mathcal{H}^{1}\left(P_{\theta}(K)\right)>0$ then
\[
P_{\theta}^{*}\left(\mathcal{H}^{1}\vert_{K}\right)=\frac{\mathcal{H}^{1}\left(K\right)}{\mathcal{H}^{1}\left(P_{\theta}(K)\right)}\cdot\mathcal{H}^{1}\vert_{P_{\theta(K)}}.
\]
\end{thm}

We establish an invariance result of the moment of inertia of self-similar
sets with several interval projections. The conclusion is that if
we rotate the self-similar set about the centre of mass by an arbitrary
angle then the moment of inertia does not change with respect to a
fixed axis.
\begin{thm}
\label{thm:SS moment of inertia}Let $K\subseteq\mathbb{R}^{2}$ be
a self-similar $1$-set. Assume that there are three different lines
through the origin such that the orthogonal projection of $K$ onto
them is an interval of length $c$ centered at the origin. Then for
$\mu=\mathcal{H}^{1}\vert_{K}$ we have that $\int x\cdot y\mathrm{d}\mu(x,y)=0$
and $\int\left|P_{\gamma}(x,y)\right|^{2}\mathrm{d}\mu(x,y)=\mathcal{H}^{1}\left(K\right)\cdot c^{2}/12$
for every $\gamma\in\mathbb{R}$. Furthermore, if $P_{\theta}(K)$
is an interval for some $\theta\in\mathbb{R}$ then $P_{\theta}(K)$
is an interval of length $c$ centered at the origin.
\end{thm}

Theorem \ref{thm:SS moment of inertia} follows from Theorem \ref{thm:vet meaasure SS},
Corollary \ref{cor:minden vet measure=00003D} and the fact that $\int_{0}^{a}x^{2}\mathrm{d}x=a^{3}/3$.

\section{Rectifiability of self-similar sets\label{sec:Rectifiability-of-self-similar}}

In this section we generalise a result of Mattila \cite[Propoition 4.2, Corollary 4.3]{Mattila-rectifiablility}
on the rectifiability of self-similar sets as we remove the separation
condition from the assumptions. Proof follows Mattila`s proof hence
we only indicate the differences in the proof.

For $K\subseteq\mathbb{R}^{d}$, $a\in K$, $s\in\mathbb{R}$ we denote
the $s$-dimensional upper density of $K$ in $a$ by
\[
\Theta^{*s}(K,a)=\limsup_{r\rightarrow0}\frac{\mathcal{H}^{s}(K\cap B(a,r))}{(2r)^{s}}
\]
and the $s$-dimensional lower density of $K$ in $a$ by
\[
\Theta_{*}^{s}(K,a)=\liminf_{r\rightarrow0}\frac{\mathcal{H}^{s}(K\cap B(a,r))}{(2r)^{s}}.
\]

\begin{defn}
\label{Def: Weak tanget def}Let $K\subseteq\mathbb{R}^{d}$ and $0\leq s\leq d$.
We say that $K$ is weakly $(s,l)$-tangential at a point $a\in\mathbb{R}^{d}$
if $\Theta^{*s}(K,a)>0$ and there is an affine $l$-plane $M$ such
that $a\in M$ and for every $\delta>0$
\[
\liminf_{r\rightarrow0}r^{-s}\cdot\mathcal{H}^{s}\left(\left(E\cap B(a,r)\right)\setminus\left\{ x:\mathrm{dist}(x,M)\leq\delta\right\} \right)=0.
\]
Then $M$ is called a weak $(s,l)$-tangent plane of $K$ at $a$.
\end{defn}

\begin{prop}
\label{prop: Tangent plane prop}Let $K$ be a self-similar set and
$s=\dim_{H}(K)$. Suppose that $\mathcal{H}^{s}(K)>0$ and $K$ has
a weak tangent plane $M$ at some point $a\in K$. Then $K\subseteq M$.
\end{prop}

Proposition \ref{prop: Tangent plane prop} was proved by Mattila
\cite[Theorem 4.2]{Mattila-rectifiablility} in the case when the
`open set condition' is satisfied. The same proof, that was used by
Mattila, can be applied to prove Proposition \ref{prop: Tangent plane prop}
in the general case. The only points where the proof uses the `open
set condition' are that $0<\mathcal{H}^{s}(K)<\infty$ and a result
of Hutchinson (\cite[Theorem 5.3 (1) (i)]{Hutchinson}) that states
that there exists $\lambda>0$ such that $\Theta_{*}^{s}(K,a)\geq\lambda$
for all $a\in K$. By an application of the implicit methods \cite[Thm 3.2]{Falconer Techniques}
we can deduce that $\mathcal{H}^{s}(K)<\infty$, and by assumption
$\mathcal{H}^{s}(K)>0$. To prove, that there exists $\lambda>0$
such that $\Theta_{*}^{s}(K,a)\geq\lambda$ for all $a\in K$, one
can follow Hutchinson's proof that only depends on the facts that
$0<\mathcal{H}^{s}(K)<\infty$ and if $r_{min}=\min\left\{ r_{i}:i\in\mathcal{I}\right\} $
then $B(a,r)$ contains $K_{\boldsymbol{\mathbf{i}}}$ for some $\boldsymbol{\mathbf{i}}\in\bigcup_{k=1}^{\infty}\mathcal{I}^{k}$
such that $r\cdot r_{min}\leq r_{\boldsymbol{\mathbf{i}}}\cdot\mathrm{diam}(K)\leq r$
for every small $r$ and $a\in K$. We leave for the reader to check
the details.

\begin{cor}
\label{cor: manifold intersection cor}Let $K$ be a self-similar
set, let $s=\dim_{H}(K)$ and $l\in\mathbb{N}$, $l\geq s$. Then
either $K\subseteq M$ for some $l$-dimensional affine subspace $M$
or $\mathcal{H}^{s}(K\cap\Gamma)=0$ for every  $l$-dimensional $C^{1}$
submanifold $\Gamma$ of $\mathbb{R}^{d}$.
\end{cor}

Corollary \ref{cor: manifold intersection cor} can be deduced from
Proposition \ref{prop: Tangent plane prop} exactly the same way as
\cite[Corollary 4.3]{Mattila-rectifiablility} is deduced from \cite[Theorem 4.2]{Mattila-rectifiablility}.

\section{Isolated interval projections\label{sec:Isolated-interval-projections}}

This section provides the main tool, Lemma \ref{lem: Key lemma},
to prove that an interval projection is isolated in the set of projections.
At the end of the section, Lemma \ref{lem: Compact IP lem}, we show
that the set of interval projections is compact.

For a set $H\subseteq\mathbb{R}^{d}$ we denote the convex hull of
$H$ by $Conv(H)$ and if $F\subseteq\mathbb{R}^{d}$ we denote the
distance between $H$ and $F$ by $\mathrm{dist}(H,F)=\inf_{x\in H,y\in F}\left\Vert x-y\right\Vert $.
On the plane the set of all lines through the origin $G_{2,1}$ can
be parameterized as $G_{2,1}=\left\{ L_{\theta}:\theta\in\nicefrac{\mathbb{R}}{\pi\mathbb{Z}}\right\} $.
Let $\prod_{2,1}$ denote the set of all orthogonal projections onto
the lines through the origin. Then $\prod_{2,1}=\left\{ \Pi_{L}:L\in G_{2,1}\right\} =\left\{ P_{\theta}:\theta\in\nicefrac{\mathbb{R}}{\pi\mathbb{Z}}\right\} $
hence $\prod_{2,1}$ inherits a topology and with this topology $\prod_{2,1}$
is compact. For a set $K\subseteq\mathbb{R}^{2}$ we denote the set
of all interval projections of $K$ by $IP(K)$.
\begin{lem}
\label{lem: Hosszu lemma 1}Let $K\subseteq\mathbb{R}^{2}$ be a compact
set, $A\subseteq K$ be a compact subset such that $IP(K)\subseteq IP(A)$
and let $L^{1}$ and $L^{2}$ be lines parallel to the $y$-axis such
that $A$ stays between $L^{1}$ and $L^{2}$, both $L^{1}$ and $L^{2}$
intersect $A$ and the $x$-coordinate of the points of $L^{1}$ is
smaller then the $x$-coordinate of the points of $L^{2}$. Let $B\subseteq K$
be a compact subset such that for every $P_{\theta}\in IP(K)$ we
have that $\mathcal{H}^{1}\left(P_{\theta}(A)\cap P_{\theta}(B)\right)=0$
and for every open set $U$ that intersects $B$ and for every $P_{\theta}\in IP(K)$
we have that $0<\mathcal{H}^{1}(P_{\theta}(B\cap U))$.

i) If either $B$ contains a point $z$ of $L^{1}$ such that there
exists $w\in A\cap L^{1}$ with $y$-coordinate smaller than the $y$-coordinate
of $z$ or $B$ contains a point $z$ of $L^{2}$ such that there
exists $w\in A\cap L^{2}$ with $y$-coordinate greater than the $y$-coordinate
of $z$, then there exists $\delta>0$ such that for every $0<\theta<\delta$
we have that $P_{\theta}\notin IP(K)$.

ii) If either $B$ contains a point $z$ of $L^{1}$ such that there
exists $w\in A\cap L^{1}$ with $y$-coordinate greater than the $y$-coordinate
of $z$ or $B$ contains a point $z$ of $L^{2}$ such that there
exists $w\in A\cap L^{2}$ with $y$-coordinate smaller than the $y$-coordinate
of $z$, then there exists $\delta>0$ such that for every $-\delta<\theta<0$
we have that $P_{\theta}\notin IP(K)$.
\end{lem}

\begin{proof}
\textit{i)} Assume that the case $w,z\in L^{1}$ holds (the proof
of the other case, when $w,z\in L^{2}$ holds, goes similarly to the
proof of this case or alternatively can be deduced from this case
by rotating everything around the origin by $\pi$). Let $w_{2}\in A\cap L^{2}$
and let $\delta\in(0,\pi)$ be such that if we draw a line $L^{3}$
through $z$ and $w_{2}$ then $L^{3}$ is parallel to $L_{\frac{\pi}{2}+\delta}$.
We claim that $P_{\theta}\notin IP(K)$ for every $0<\theta<\delta$.
Assume for a contradiction that $P_{\theta}\in IP(K)$ for some $0<\theta<\delta$.
Let $L^{4}$ be the line through $z$ such that $L^{4}$ is parallel
to $L_{\frac{\pi}{2}+\theta}$. Then $L^{4}$ lies strictly between
the points $w$ and $w_{2}$ hence $P_{\theta}(z)$ is in the interior
of the interval $Conv\left(\left\{ P_{\theta}(w),P_{\theta}(w_{2})\right\} \right)$
(see Figure 1). Since $w,w_{2}\in A$ and by assumption $P_{\theta}\in IP(K)\subseteq IP(A)$
it follows that $P_{\theta}(A)$ is an interval and $Conv\left(\left\{ P_{\theta}(w),P_{\theta}(w_{2})\right\} \right)\subseteq P_{\theta}(A)$.
Thus $P_{\theta}(z)$ is in the interior of the interval $P_{\theta}(A)$.
Let $U$ be an open neighborhood of $z$ such that $P_{\theta}(U)\subseteq P_{\theta}(A)$.
Then $P_{\theta}(B\cap U)\subseteq P_{\theta}(A)\cap P_{\theta}(B)$
and by assumption $0<\mathcal{H}^{1}(P_{\theta}(B\cap U))$. Hence
$0<\mathcal{H}^{1}\left(P_{\theta}(A)\cap P_{\theta}(B)\right)$ but
this contradicts the assumption of the lemma.

The proof of \textit{ii)} goes similarly to the proof of \textit{i)}
or alternatively \textit{ii)} can be deduced from \textit{i)} by reflecting
everything in the $y$-axis.
\end{proof}
\includegraphics[scale=0.4]{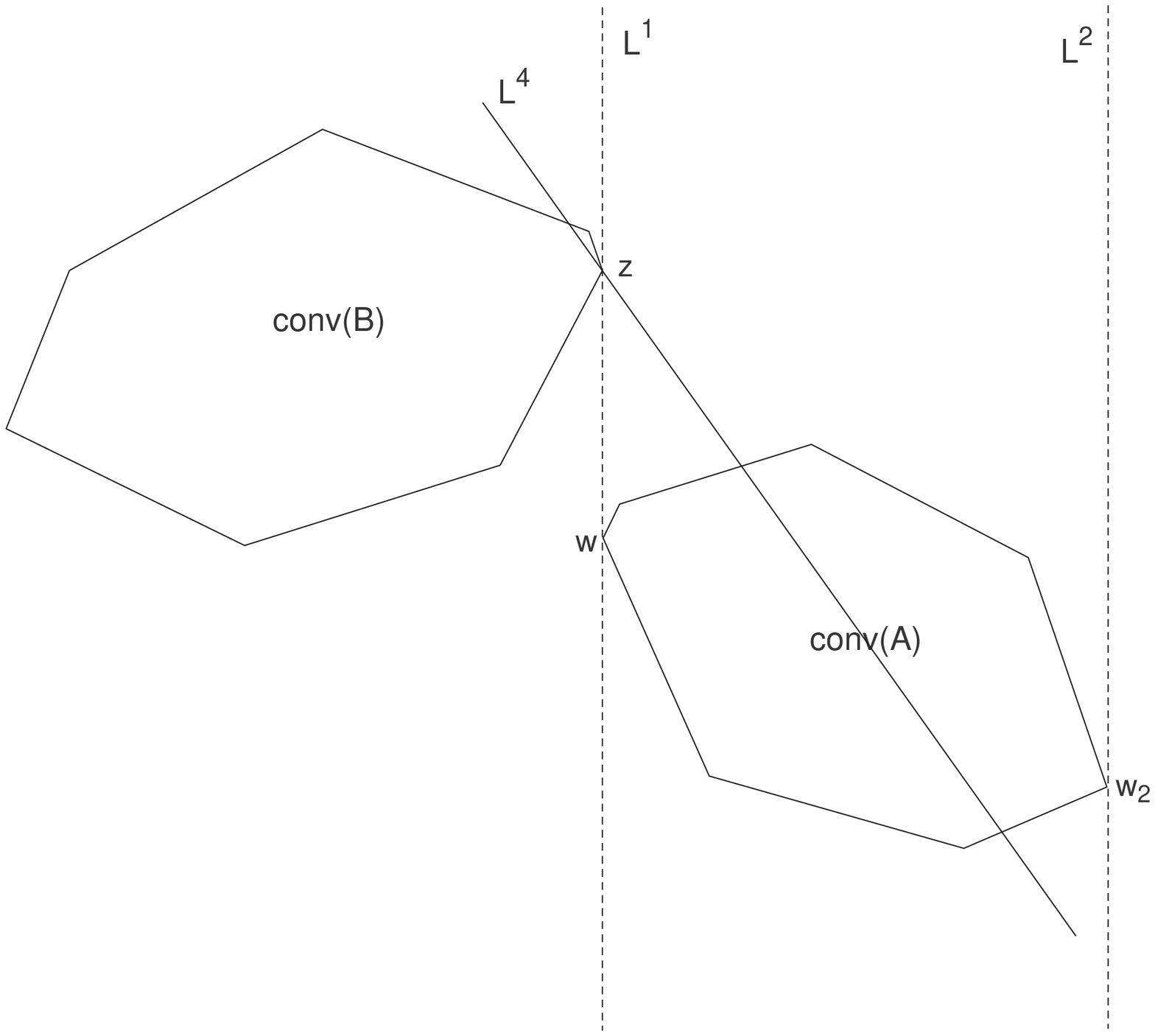}~~~~~~~~~~~~~\includegraphics[scale=0.4]{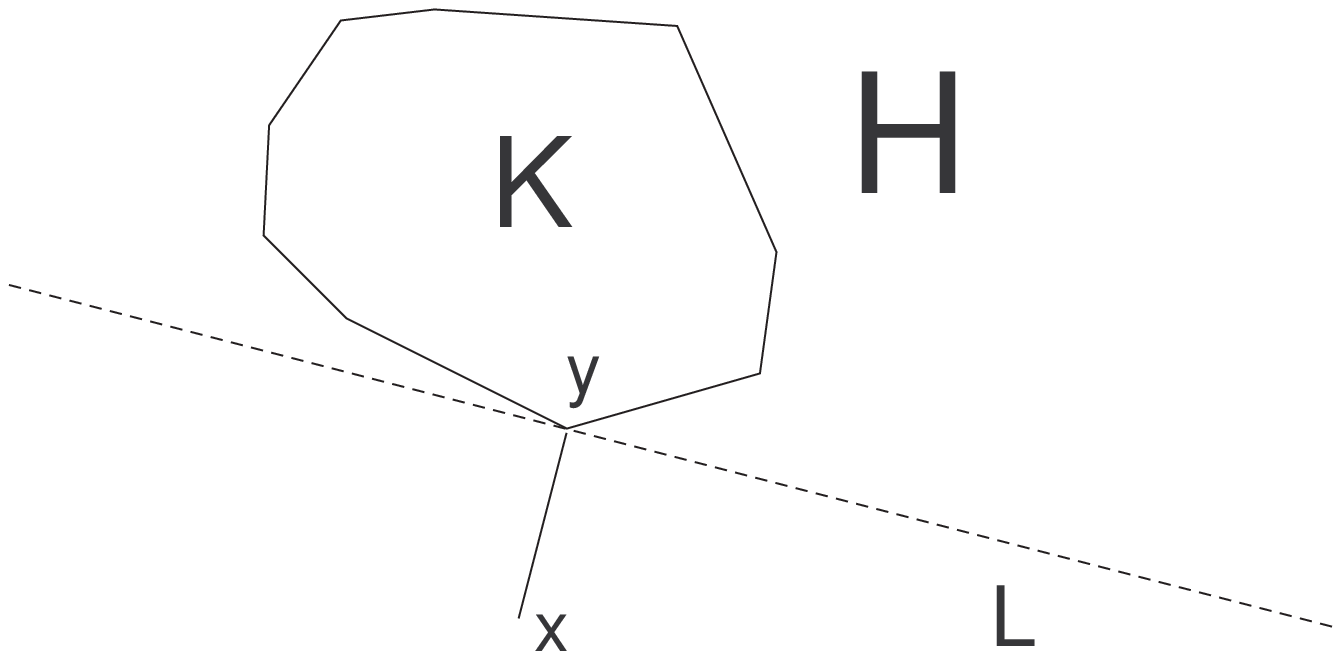}

~~~~~~~~~~~~~~~~~~~~~~~~Figure 1.~~~~~~~~~~~~~~~~~~~~~~~~~~~~~~~~~~~~~~~~~~~~~~~~~~~~~~~~~~~Figure
2.
\begin{lem}
\label{lem:convex perpendicular}Let $K\subseteq\mathbb{R}^{2}$ be
a convex, compact set and let $x\in\mathbb{R}^{2}\setminus K$ and
$y\in K$ such that $\mathrm{dist}(\{x\},K)=\left\Vert x-y\right\Vert $.
Let $L$ be the line through $y$ that is perpendicular to the line
segment $[x,y]$ and let $H$ be the closed half-plane that is bordered
by $L$ and $x\notin H$. Then $K\subseteq H$.
\end{lem}

The proof is trivial geometric argument, that if $z\in K\setminus H$
then there is a closer point to $x$ than $y$ on the line segment
$[y,z]$. See Figure 2.
\begin{lem}
\label{lem: Hosszu lemma 2}Let $K\subseteq\mathbb{R}^{2}$ be a compact
set. Let $A,B\subseteq K$ be compact subsets such that $Conv(A)\cap Conv(B)=\emptyset$,
$P_{0}\in IP(A)$, $P_{0}\in IP(B)$, $\mathrm{diam}(P_{0}(A))>0$,
$\mathrm{diam}(P_{0}(B))>0$ and $\left|P_{0}(A)\cap P_{0}(B)\right|=1$.
Then $P_{0}\in IP(A\cup B)$. Assume that $P_{0}(K\setminus(A\cup B))$
and the interior of the interval $P_{0}(A\cup B)$ are disjoint. Let
$L$ be the line through $P_{0}(A)\cap P_{0}(B)$ that is parallel
to the $y$-axis.

i) If either $A$ is to the right of $L$ and there exist $w\in A\cap L$
and $z\in B\cap L$ such that the $y$-coordinate of $w$ is smaller
than the $y$-coordinate of $z$ or $A$ is to the left of $L$ and
there exist $w\in A\cap L$ and $z\in B\cap L$ such that the $y$-coordinate
of $w$ is greater than the $y$-coordinate of $z$, then there exists
$\delta>0$ such that for every $-\delta<\theta<0$ we have that $P_{\theta}\notin IP(K)$.

ii) If either $A$ is to the right of $L$ and there exist $w\in A\cap L$
and $z\in B\cap L$ such that $y$-coordinate of $w$ is greater than
the $y$-coordinate of $z$ or $A$ is to the left of $L$ and there
exist $w\in A\cap L$ and $z\in B\cap L$ such that the $y$-coordinate
of $w$ is smaller than the $y$-coordinate of $z$, then there exists
$\delta>0$ such that for every $0<\theta<\delta$ we have that $P_{\theta}\notin IP(K)$.
\end{lem}

\begin{proof}
\textit{i)} Assume that the case, $A$ is to the right of $L$, holds
(the proof of the other case goes similarly to the proof of this case
or alternatively can be deduced from this case by rotating everything
around the origin by $\pi$). Since $P_{0}(A)$ and $P_{0}(B)$ are
intervals that intersect we have that $P_{0}(A\cup B)$ is an in interval,
so $P_{0}\in IP(A\cup B)$.

Since $A$ and $B$ are compact $Conv(A)$ and $Conv(B)$ are compact
sets. By assumption $Conv(A)\cap Conv(B)=\emptyset$. Hence $\alpha=\mathrm{dist}(Conv(A),Conv(B))>0$.
By compactness there exist $g\in B$, $h\in A$ such that $\left\Vert h-g\right\Vert =\alpha$.
Let $L^{5}$ be the perpendicular bisector of the line segment $[g,h]$
and let $\delta_{1}\in(0,\pi)$ such that $L_{\frac{\pi}{2}-\delta_{1}}$
is parallel to $L^{5}$ (note that we can choose such $\delta_{1}\in(0,\pi)$
since $L^{5}$ is not parallel to the $y$-axis as the convex sets
$Conv(A)$ and $Conv(B)$ both intersect $L$ but stay on different
sides of $L$). So $\left|L\cap L^{5}\right|=1$ and let $p\in\mathbb{R}^{2}$
be the point of intersection $L\cap L^{5}$. By Lemma \ref{lem:convex perpendicular}
$Conv(A)$ stays below $L^{5}$ and $Conv(B)$ stays above $L^{5}$.

Since $\left|P_{0}(A)\cap P_{0}(B)\right|=1$ the compact intervals
$P_{0}(A)$ and $P_{0}(B)$ intersect in one point and $P_{0}(A)$
is to the right of the intersection. Let $a<b<c$ on the line $\mathbb{R}\times\{0\}=\mathbb{R}$
such that $P_{0}(B)=[a,b]$ and $P_{0}(A)=[b,c]$ and hence $P_{0}(L)=\{b\}$.
Since $K$ is compact there exists $R>0$ such that $K$ is contained
in the square $[-R,R]\times[-R,R]$. In particular, $A\cup B$ is
contained in the rectangle $[a,c]\times[-R,R]$. Let $\delta_{2}\in(0,\pi)$
be such that for every $-\delta_{2}<\theta<0$ the line through $p$,
that is parallel to $L_{\frac{\pi}{2}+\theta}$, intersects both $[a,b]\times\{-R\}$
and $[b,c]\times\{R\}$.

Let $\delta=\min\{\delta_{1},\delta_{2}\}$. We claim that $P_{\theta}\notin IP(K)$
for every $-\delta<\theta<0$. Let $-\delta<\theta<0$ be arbitrary
and let $L^{6}$ be the line through $p$ that is parallel to $L_{\frac{\pi}{2}+\theta}$.
By the choice of $\delta_{1}$ the line $L^{6}$ avoids $Conv(A)$
and $Conv(B)$ and hence avoids $A$ and $B$ (see Figure 3). By the
assumption, that $P_{0}(K\setminus(A\cup B))$ and the interior of
the interval $P_{0}(A\cup B)$ are disjoint, it follows that $K\setminus(A\cup B)$
does not intersect $(a,c)\times\mathbb{R}$ but by the choice of $\delta_{2}$
we have that $L^{6}$ intersects $[-R,R]\times[-R,R]$ only inside
$(a,c)\times\mathbb{R}$ (see Figure 4). Hence by the choice of $R$
we have that $K\setminus(A\cup B)$ does not intersect $L^{6}$. So
$L^{6}$ does not intersect $A$, $B$ and $K\setminus(A\cup B)$
thus $L^{6}$ does not intersect $K$. Hence $P_{\theta}(p)\notin P_{\theta}(K)$.
But since $Conv(A)$ stays below $L^{6}$ and $Conv(B)$ stays above
$L^{6}$ and so $A$ stays below $L^{6}$ and $B$ stays above $L^{6}$
we have that $P_{\theta}(p)\in Conv(P_{\theta}(K))$. Thus $P_{\theta}(K)$
is not an interval.

The proof of \textit{ii)} goes similarly to the proof of \textit{i)}
or alternatively \textit{ii)} can be deduced from \textit{i)} by reflecting
everything in the $y$-axis.
\end{proof}
\includegraphics[scale=0.4]{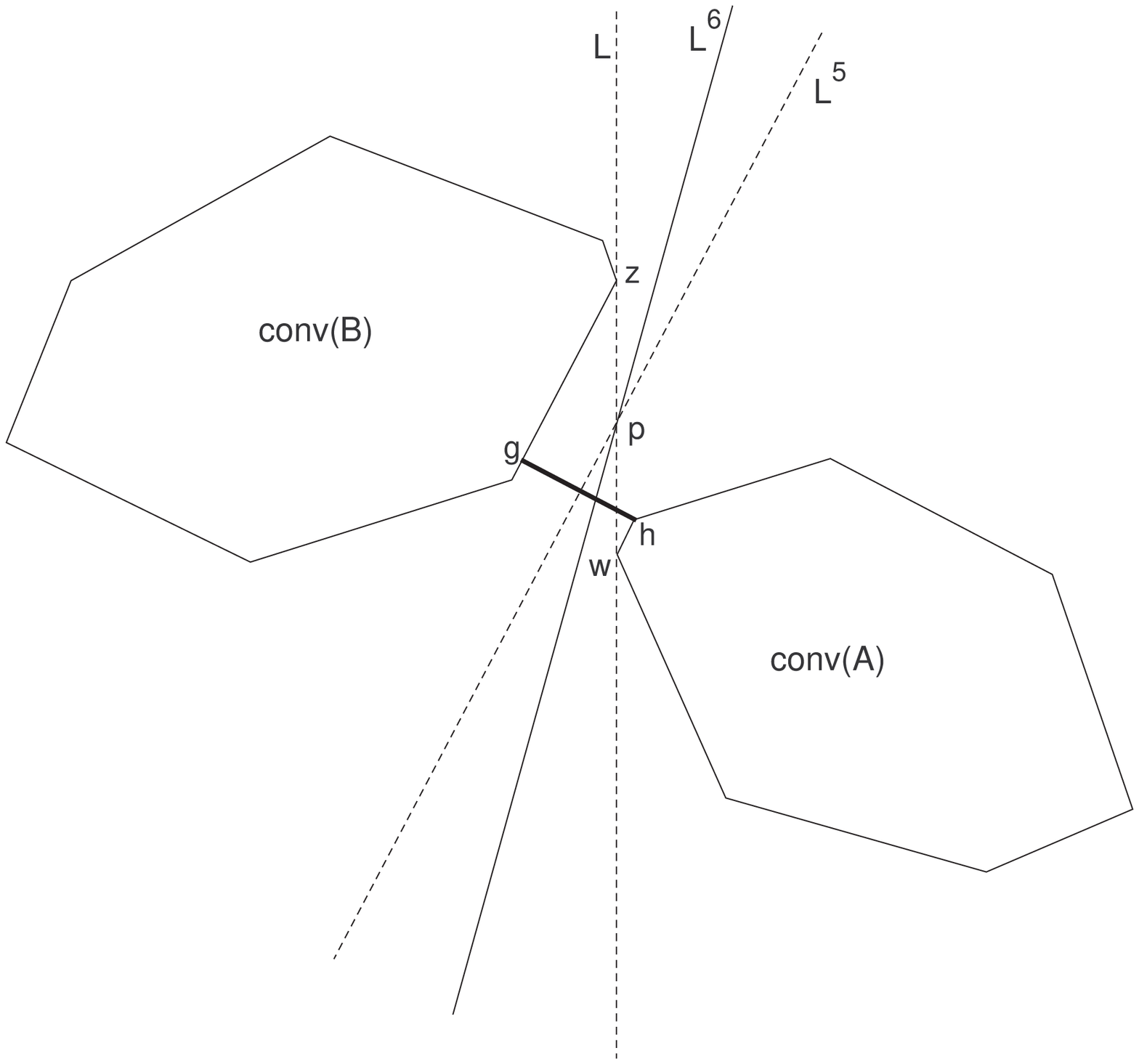}~~~~~~~~~~\includegraphics[scale=0.45]{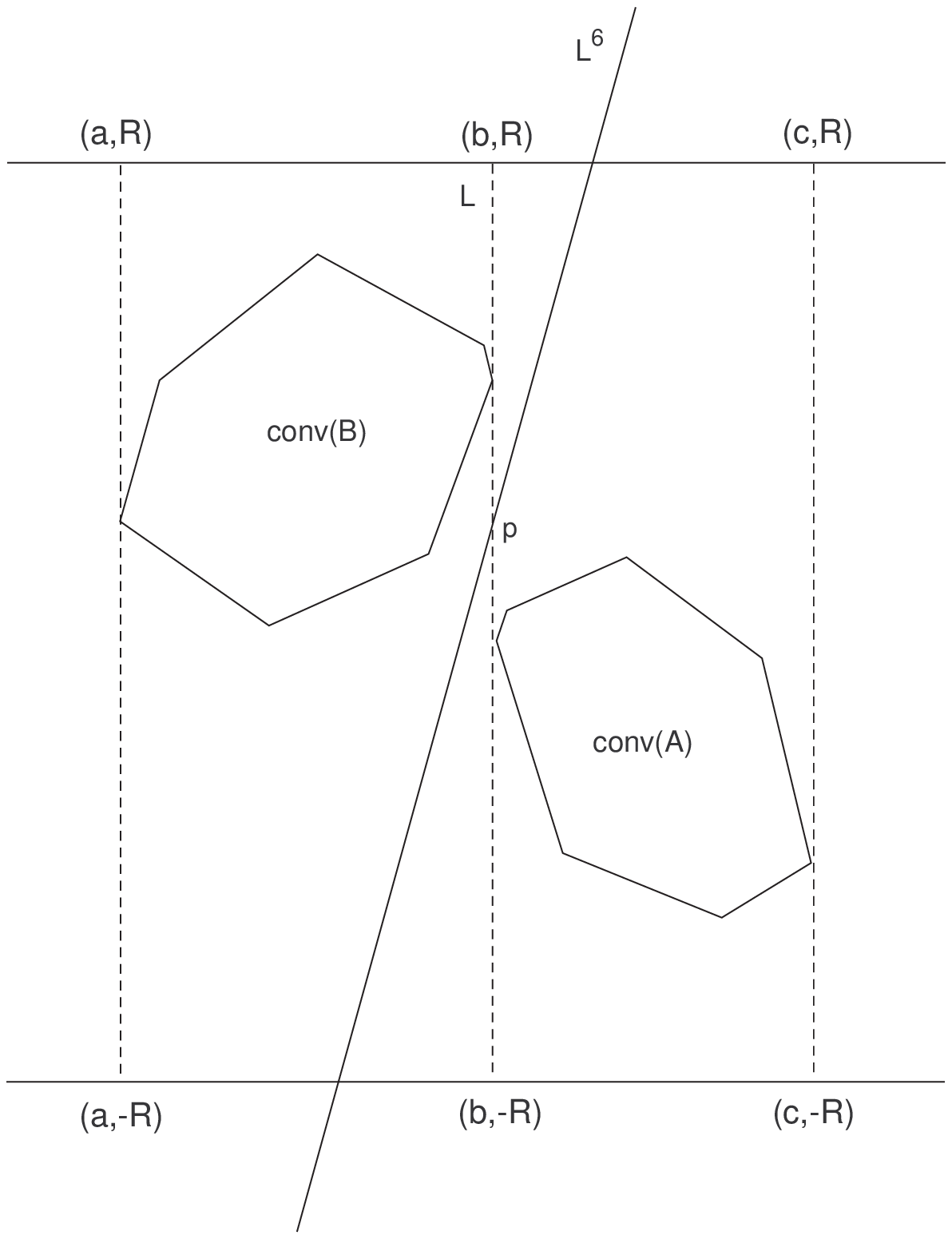}

~~~~~~~~~~~~~~~~~~~~~~~~~Figure 3.~~~~~~~~~~~~~~~~~~~~~~~~~~~~~~~~~~~~~~~~~~~~~~~~~~~~~~~~Figure
4.
\begin{lem}
\label{lem: Key lemma}Let $K\subseteq\mathbb{R}^{2}$ be a compact
set. Let $A\subseteq K$ be a compact subset such that $IP(K)\subseteq IP(A)$,
$P_{0}\in IP(A)$ and $\mathrm{diam}(P_{0}(A))>0$. Let $B\subseteq K$
be a compact subset such that for every $P_{\theta}\in IP(K)$ we
have that $\mathcal{H}^{1}\left(P_{\theta}(A)\cap P_{\theta}(B)\right)=0$,
for every open set $U$ that intersects $B$ and for every $P_{\theta}\in IP(K)$
we have that $0<\mathcal{H}^{1}(P_{\theta}(B\cap U))$, $P_{0}\in IP(B)$,
$\mathrm{diam}(P_{0}(B))>0$, $\left|P_{0}(A)\cap P_{0}(B)\right|=1$
and $A\cap B=\emptyset$. Then $P_{0}\in IP(A\cup B)$. Assume that
$P_{0}(K\setminus(A\cup B))$ and the interior of the interval $P_{0}(A\cup B)$
are disjoint. Then there exists $\delta>0$ such that for every $\theta\in(-\delta,\delta)\setminus\{0\}$
we have that $P_{\theta}\notin IP(K)$.
\end{lem}

\begin{proof}
Since $P_{0}(A)$ and $P_{0}(B)$ are intervals that intersect it
follows that $P_{0}(A\cup B)$ is an in interval and so $P_{0}\in IP(A\cup B)$.
Let $L$ be the line through $P_{0}(A)\cap P_{0}(B)$ that is parallel
to the $y$-axis. Assume that the case that $A$ is to the right of
$L$ holds (the proof of the other case goes similarly to the proof
of this case or alternatively can be deduced from this case by rotating
everything around the origin by $\pi$). Let $E_{A}=\left\{ y\in\mathbb{R}:\exists x\in\mathbb{R},(x,y)\in A\cap L\right\} $
be the set of the $y$-coordinates of the points of $A\cap L$ and
let $E_{B}=\left\{ y\in\mathbb{R}:\exists x\in\mathbb{R},(x,y)\in B\cap L\right\} $
be the set of the $y$-coordinates of the points of $B\cap L$. Let
$I_{A}=Conv(E_{A})$ and $I_{B}=Conv(E_{B})$. Since $A\cap B=\emptyset$
we have that $E_{A}$ and $E_{B}$ are disjoint compact sets. Hence
either $I_{A}$ and $I_{B}$ are disjoint intervals or $I_{A}$ and
$I_{B}$ are overlapping intervals, i.e. the interior of $I_{A}\cap I_{B}$
is nonempty. We will consider these two cases in the proof.

\textit{Case 1}: $I_{A}$ and $I_{B}$ are disjoint intervals

\textit{Case 2}: $I_{A}$ and $I_{B}$ are overlapping intervals

\textit{Case 1}: Assume that the case $\sup I_{A}<\inf I_{B}$ holds
(the proof of the other case, $\inf I_{A}>\sup I_{B}$, goes similarly
to the proof of this case or alternatively can be deduced from this
case by reflecting everything in the $x$-axis). Since $\sup I_{A}<\inf I_{B}$
with the notation $L=L^{1}$ we have that $B$ contains a point $z$
of $L^{1}$ such that there exists $w\in A\cap L^{1}$ with $y$-coordinate
smaller than the $y$-coordinate of $z$. So the conditions of Lemma
\ref{lem: Hosszu lemma 1} \textit{i)} are satisfied. Since $P_{0}(A)$
and $P_{0}(B)$ are non-overlapping intervals and $I_{A}$ and $I_{B}$
are disjoint intervals it follows that $Conv(A)\cap Conv(B)=\emptyset$.
Again since $\sup I_{A}<\inf I_{B}$ it follows that there exist a
$w\in A\cap L$ and $z\in B\cap L$ such that $y$-coordinate of $w$
is smaller than the $y$-coordinate of $z$. We assumed that $A$
is to the right of $L$. So the conditions of Lemma \ref{lem: Hosszu lemma 2}
\textit{i)} are satisfied. So the conditions of both Lemma \ref{lem: Hosszu lemma 1}
\textit{i)} and Lemma \ref{lem: Hosszu lemma 2} \textit{i)} are satisfied.
Hence there exists $\delta>0$ such that $P_{\theta}\notin IP(K)$
for every $\theta\in(-\delta,\delta)\setminus\{0\}$.

\textit{Case 2}: Since $I_{A}$ and $I_{B}$ are overlapping intervals
with the notation $L=L^{1}$ we have that $B$ contains a point $z_{1}$
of $L^{1}$ such that there exists $w_{1}\in A\cap L^{1}$ with $y$-coordinate
smaller than the $y$-coordinate of $z_{1}$, and that $B$ contains
a point $z_{2}$ of $L^{1}$ such that there exists $w_{2}\in A\cap L^{1}$
with $y$-coordinate greater than the $y$-coordinate of $z_{2}$.
So the conditions of Lemma \ref{lem: Hosszu lemma 1} \textit{i)}
are satisfied with $z=z_{1}$ and the conditions of Lemma \ref{lem: Hosszu lemma 1}
\textit{ii)} are satisfied with $z=z_{2}$. Hence there exists $\delta>0$
such that $P_{\theta}\notin IP(K)$ for every $\theta\in(-\delta,\delta)\setminus\{0\}$.
\end{proof}

\begin{lem}
\label{lem: Compact IP lem}Let $K\subseteq\mathbb{R}^{2}$ be compact.
Then $IP(K)$ is compact.
\end{lem}

\begin{proof}
Since $\prod_{2,1}$ is compact it is enough to show that $IP(K)$
is closed i.e. $\prod_{2,1}\setminus IP(K)$ is open. Let $P_{\theta}\in\prod_{2,1}\setminus IP(K)$.
We need to show that a neighborhood of $P_{\theta}$ is contained
in $\prod_{2,1}\setminus IP(K)$. Without the loss of generality we
can assume that $\theta=0$ otherwise we can rotate everything around
the origin by $-\theta$. Since $P_{0}$ is not an interval projection
of $K$ it follows that $P_{0}(K)$ is not an interval. Let $x\in Conv(P_{0}(K))\setminus P_{0}(K)$.
Since $K$ is compact it follows that $P_{0}(K)$ is compact as well
and hence there exists $r>0$ such that the interval $[x-r,x+r]$
is contained in $Conv(P_{0}(K))\setminus P_{0}(K)$. Since $K$ is
compact there exists $R>0$ such that $K$ is contained in the square
$[-R,R]\times[-R,R]$. Let $y=\inf P_{0}(K)$ and $z=\sup P_{0}(K)$.
Then $K$ is contained in $[y,x-r)\times[-R,R]\cup(x+r,z]\times[-R,R]$
and both components intersect $K$. Let $\delta>0$ be such that the
line, that connects the points $(x-r,R)$ and $(x+r,-R)$, is parallel
to $L_{\frac{\pi}{2}+\delta}$. Then for every $\varphi\in(-\delta,\delta)$
the line, that goes through the point $(x,0)$ and is parallel to
$L_{\frac{\pi}{2}+\varphi}$, does not intersect $K$ and separates
$K$ into two non-empty part. Hence $P_{\varphi}(K)$ is not an interval
for every $\varphi\in(-\delta,\delta)$.
\end{proof}

\section{Interval projections of self-similar sets\label{sec:Interval-projections-of}}

In this section we prove the main results of the paper, Theorem \ref{thm:SS int-proj SSC}
and Theorem \ref{thm:SS int-proj OSC}. Since $IP(K)$ is compact
by Lemma \ref{lem: Compact IP lem} it is enough to prove that every
element of $IP(K)$ is isolated. We proceed by finding subsets $A,B\subseteq K$
that satisfies the assumption of Lemma \ref{lem: Key lemma}.

In this section let $\left\{ S_{i}\right\} _{i=1}^{m}$ be a SS-IFS
in $\mathbb{R}^{2}$ with attractor $K$ and $S_{i}$ be decomposed
as in (\ref{eq:S_ideconposition}) and we denote $\max\left\{ r_{i}\right\} _{i=1}^{m}<1$
by $r_{max}$. We further assume that $\dim_{H}(K)=1$.

\begin{lem}
\label{lem: SS density lem}If $\mathcal{H}^{1}(P_{\theta}(K))>0$
then for every open set $U$ that intersects $K$ we have that $0<\mathcal{H}^{1}(P_{\theta}(K\cap U))$.
\end{lem}

\begin{proof}
Let $P_{\theta}\in\prod_{2,1}$ such that $\mathcal{H}^{1}(P_{\theta}(K))>0$,
let $U$ be an open set that intersects $K$ and $x\in K\cap U$.
Then there exists a ball $B$, centered at $x$ with radius $r_{U}>0$,
such that $B\subseteq U$. There exists a cylinder set $K_{\boldsymbol{\mathbf{i}}}\subseteq K\cap U$
for large enough $k$ and for some $\boldsymbol{\mathbf{i}}\in\mathcal{I}^{k}$.
By \cite[Theorem 1.3]{linim} we have that
\[
\mathcal{H}^{1}(P_{\theta}(K\cap U))\geq\mathcal{H}^{1}\left(P_{\theta}(K_{\boldsymbol{\mathbf{i}}})\right)=\frac{\mathcal{H}^{1}(K_{\boldsymbol{\mathbf{i}}})}{\mathcal{H}^{1}(K)}\cdot\mathcal{H}^{1}\left(P_{\theta}(K)\right)=r_{\boldsymbol{\mathbf{i}}}^{k}\cdot\mathcal{H}^{1}\left(P_{\theta}(K)\right)>0.
\]
\end{proof}

\begin{lem}
\label{lem: SS Almost disjoint proj}Let $A,B\subseteq K$ be disjoint
compact subsets. Then $\mathcal{H}^{1}\left(P_{\theta}(A)\cap P_{\theta}(B)\right)=0$
for every $P_{\theta}\in\prod_{2,1}$.
\end{lem}

For details of the proof see \cite[Corolllary 1.4]{linim}.\\

\noindent \textit{Proof of Theorem \ref{thm:SS int-proj SSC}.} We
need to show that $IP(K)$ is finite. Since $K$ is compact, $IP(K)$
is a compact subset of $\prod_{2,1}$ by Lemma \ref{lem: Compact IP lem}.
Thus it is enough to show that every $P_{\theta}\in IP(K)$ is isolated
in $IP(K)$. If $IP(K)=\emptyset$ the proof is trivial, so we assume
that $IP(K)\neq\emptyset$. Hence by Theorem \ref{thm:infinite T no IP}
we have that $\left|\mathcal{T}\right|<\infty$. If $K$ is contained
in a line, that is parallel to $L_{\frac{\pi}{2}+\gamma}$ for some
$\gamma\in[0,\pi)$, then $\left|IP(K)\right|=1$ because $\left\{ S_{i}\right\} _{i=1}^{m}$
satisfies the SSC and so the only interval projection of $K$ is $P_{\gamma}$.
So we assume that $q=\left|\mathcal{T}\right|<\infty$ and $K$ is
not contained in any line.

\noindent Let $P_{\theta}\in IP(K)$ be arbitrary. We need to show
that there exists $\delta>0$ such that $P_{\theta+\varphi}\notin IP(K)$
for every $\varphi\in(-\delta,\delta)\setminus\{0\}$. Without the
loss of generality we can assume that $\theta=0$ otherwise we can
rotate everything around the origin by $-\theta$. Let $i_{1}\in\mathcal{I}$
be arbitrary and $\boldsymbol{\mathbf{i}}=\left(i_{1},\ldots,i_{1}\right)\in\mathcal{I}^{2\cdot q}$.
Then $T_{\boldsymbol{\mathbf{i}}}=T_{i_{1}}^{2\cdot q}$ is the the
identity map because $q=\left|\mathcal{T}\right|<\infty$. So $K_{\boldsymbol{\mathbf{i}}}=r_{\boldsymbol{\mathbf{i}}}\cdot K+t_{\boldsymbol{\mathbf{i}}}$
and hence $IP(K_{\boldsymbol{\mathbf{i}}})=IP(K)$. Since $P_{0}\in IP(K)$
let
\[
P_{0}(K)=[a,b]\subseteq\mathbb{R}\times\{0\}=\mathbb{R}
\]
and
\[
P_{0}(K_{\boldsymbol{\mathbf{i}}})=[c,d]\subseteq\mathbb{R}\times\{0\}=\mathbb{R}
\]
and both $[a,b]$ and $[c,d]$ are intervals of positive length because
$K$ is not contained in any line. Then either $[a,c]$ or $[d,b]$
is an interval of positive length. Let this interval of positive length
be $J$, let $B=\left(P_{0}^{-1}(J)\cap K\right)\setminus K_{\boldsymbol{\mathbf{i}}}$
and $A=K_{\boldsymbol{\mathbf{i}}}$. Then $B$ is compact by that
$\left\{ S_{i}\right\} _{i=1}^{m}$ satisfies the SSC. So $P_{0}\in IP(K)=IP(A)$,
$P_{0}\in IP(B)$, $A$ and $B$ are disjoint compact subsets of $K$,
$\left|P_{0}(A)\cap P_{0}(B)\right|=1$, $\mathrm{diam}(P_{0}(A))>0$,
$\mathrm{diam}(P_{0}(B))>0$, $P_{0}\in IP(A\cup B)$, $P_{0}(K\setminus(A\cup B))$
and the interior of the interval $P_{0}(A\cup B)$ are disjoint and
by Lemma \ref{lem: SS Almost disjoint proj} for every $P_{\theta}\in\prod_{2,1}$
we have that $\mathcal{H}^{1}\left(P_{\theta}(A)\cap P_{\theta}(B)\right)=0$.
By Lemma \ref{lem: SS density lem} for every open set $U$ that intersects
$B$ and for every $P_{\theta}\in IP(K)$ we have that $0<\mathcal{H}^{1}(P_{\theta}(B\cap U))$.
So all the conditions of Lemma \ref{lem: Key lemma} are satisfied
for $A,B\subseteq K$. Hence there exists $\delta>0$ such that $P_{0+\varphi}\notin IP(K)$
for every $\varphi\in(-\delta,\delta)\setminus\{0\}$. So $P_{0}\in IP(K)$
is isolated in $IP(K)$.$\hfill\square$

\begin{lem}
\label{lem:gorbe resz lem}Let $K\subseteq\mathbb{R}^{2}$ be such
that $\mathcal{H}^{1}(K)<\infty$. If there exists a non-degenerate
continuous curve $\Gamma_{0}$ such that $\Gamma_{0}\subseteq K$
then there exists a $1$-dimensional $C^{1}$ submanifold $\Gamma$
of $\mathbb{R}^{2}$ such that $\mathcal{H}^{1}(K\cap\Gamma)>0$.
\end{lem}

\begin{proof}
If $x,y\in\Gamma_{0}$, $x\neq y$ then $P_{\theta}(K)$ is an interval
of positive length for each $\theta$ such that $L_{\theta}$ is not
perpendicular to the line segment $[x,y]$. So the statement follows
from \cite[Theorem 18.1 (2)]{Mattilakonyv}.
\end{proof}

\noindent \textit{Proof of Theorem \ref{thm:SS int-proj OSC}.} As
in the proof of Theorem \ref{thm:SS int-proj SSC} it is enough to
show that for $P_{\theta}\in IP(K)$ there exists $\delta>0$ such
that $P_{\theta+\varphi}\notin IP(K)$ for every $\varphi\in(-\delta,\delta)\setminus\{0\}$
and without the loss of generality we can assume that $\theta=0$.
We proceed by showing that there are sets $A,B\subseteq K$ that satisfies
the assumptions of Lemma \ref{lem: Key lemma} otherwise $K$ would
contain a continuous curve which would contradict with the unrectifiability
of $K$.

Since each $S_{i}$ is a homothety for all $i\in\mathcal{I}$ it follows
that $IP(K_{\boldsymbol{\mathbf{i}}})=IP(S_{\boldsymbol{\mathbf{i}}}(K))=IP(K)$
for all $k\in\mathbb{N}$ and $\boldsymbol{\mathbf{i}}\in\mathcal{I}^{k}$.
Since $P_{0}\in IP(K)$ let
\[
P_{0}(K)=[a,b]\subseteq\mathbb{R}\times\{0\}=\mathbb{R}.
\]
So $[a,b]$ is an interval of positive length because $K$ is not
contained in any line.

We claim that there is a unique $w\in K$ such that $P_{0}(w)=a$.
Otherwise, assume for a contradiction that $w'$ is another point
of $K$ such that $P_{0}(w')=a$. Let $r_{d}=\left\Vert w-w'\right\Vert $
and $k\in\mathbb{N}$ such that $r_{max}^{k}\cdot\mathrm{diam}(K)<\frac{r_{d}}{2}$.
Then $\mathrm{diam}(K_{\boldsymbol{\mathbf{i}}})<r_{max}^{k}\cdot\mathrm{diam}(K)<\frac{r_{d}}{2}$
for every $\boldsymbol{\mathbf{i}}\in\mathcal{I}^{k}$. There exist
$\boldsymbol{\mathbf{i}},\boldsymbol{\mathbf{j}}\in\mathcal{I}^{k}$
such that $w\in K_{\boldsymbol{\mathbf{i}}}$ and $w'\in K_{\boldsymbol{\mathbf{j}}}$.
Since $\mathrm{diam}(K_{\boldsymbol{\mathbf{i}}})<\frac{r_{d}}{2}$,
$w\in K_{\boldsymbol{\mathbf{i}}}$ and $\mathrm{diam}(K_{\boldsymbol{\mathbf{j}}})<\frac{r_{d}}{2}$,
$w'\in K_{\boldsymbol{\mathbf{j}}}$ it follows that $K_{\boldsymbol{\mathbf{i}}}$
and $K_{\boldsymbol{\mathbf{j}}}$ are disjoint. Since $P_{0}\in IP(K_{\boldsymbol{\mathbf{i}}})$
and $P_{0}\in IP(K_{\boldsymbol{\mathbf{j}}})$ we have that for some
$c_{1},c_{2}>a$ the projections $P_{0}(K_{\boldsymbol{\mathbf{i}}})=[a,c_{1}]$
and $P_{0}(K_{\boldsymbol{\mathbf{j}}})=[a,c_{2}]$ are overlapping
intervals. On the other hand, this contradicts with Lemma \ref{lem: SS Almost disjoint proj}.
So there is a unique $w\in K$ such that $P_{0}(w)=a$. Similarly
there is a unique $z\in K$ such that $P_{0}(z)=b$.

So $P_{0}(K_{\boldsymbol{\mathbf{i}}})$ is an interval and for all
$\boldsymbol{\mathbf{i}}\in\mathcal{I}^{k}$ and for both endpoints
of the interval $P_{0}(K_{\boldsymbol{\mathbf{i}}})$ there is a unique
point of $K_{\boldsymbol{\mathbf{i}}}$ that projects onto that endpoint
of the interval $P_{0}(K_{\boldsymbol{\mathbf{i}}})$ because each
$S_{\boldsymbol{\mathbf{i}}}$ is a homothety. We say that a cylinder
set, $K_{\boldsymbol{\mathbf{i}}}$ for some $\boldsymbol{\mathbf{i}}\in\mathcal{I}^{k}$,
is a fitting piece if there exist unique $w_{\boldsymbol{\mathbf{i}}},z_{\boldsymbol{\mathbf{i}}}\in K$
such that $P_{0}(w_{\boldsymbol{\mathbf{i}}})$ and $P_{0}(z_{\boldsymbol{\mathbf{i}}})$
are the two endpoints of the interval $P_{0}(K_{\boldsymbol{\mathbf{i}}})$.
We claim that there exists $\boldsymbol{\mathbf{i}}\in\mathcal{I}^{k}$
for some $k\in\mathbb{N}$ such that $K_{\boldsymbol{\mathbf{i}}}$
is not a fitting piece. Assume for a contradiction that $K_{\boldsymbol{\mathbf{i}}}$
is a fitting piece for each $k\in\mathbb{N}$, $\boldsymbol{\mathbf{i}}\in\mathcal{I}^{k}$
and without the loss of generality we can assume that the $x$-coordinate
of $w_{\boldsymbol{\mathbf{i}}}$ is smaller than the $x$-coordinate
of $z_{\boldsymbol{\mathbf{i}}}$. Let $a=x_{0,k}<x_{1,k}<\ldots<x_{n_{k},k}=b$
such that $\left\{ x_{j,k}:j\in\left\{ 0,1,\ldots,n_{k}\right\} \right\} =\left\{ P_{0}(w_{\boldsymbol{\mathbf{i}}}):\boldsymbol{\mathbf{i}}\in\mathcal{I}^{k}\right\} \bigcup\left\{ P_{0}(z_{\boldsymbol{\mathbf{i}}}):\boldsymbol{\mathbf{i}}\in\mathcal{I}^{k}\right\} $.
For $k\in\mathbb{N}$ let $f_{k}:[a,b]\longrightarrow\mathbb{R}$
be the continuous function such that $w_{\boldsymbol{\mathbf{i}}},z_{\boldsymbol{\mathbf{i}}}\in graph(f_{k})$
for each $\boldsymbol{\mathbf{i}}\in\mathcal{I}^{k}$ and the restriction
of $f_{k}$ to the interval $[x_{j,k},x_{j+1,k}]$ is linear. If $P_{0}(w_{\boldsymbol{\mathbf{j}}})\in P_{0}(K_{\boldsymbol{\mathbf{i}}})$
for some $\boldsymbol{\mathbf{i}},\boldsymbol{\mathbf{j}}\in\mathcal{I}^{k}$
then $w_{\boldsymbol{\mathbf{j}}}\in K_{\boldsymbol{\mathbf{i}}}$
because $K_{\boldsymbol{\mathbf{j}}}$ is a fitting piece. Thus if
$P_{0}(K_{\boldsymbol{\mathbf{i}}})\cap P_{0}(K_{\boldsymbol{\mathbf{j}}})\neq\emptyset$
then the $2\cdot r_{max}^{k}\cdot\mathrm{diam}(K)$-neighbourhood
of $w_{\boldsymbol{\mathbf{i}}}$ contains $K_{\boldsymbol{\mathbf{j}}}$
and so also contains $f_{k}(P_{0}(K_{\boldsymbol{\mathbf{i}}}))$.
Hence $\left|f_{k}(x)-f_{k+l}(x)\right|\leq4\cdot r_{max}^{k}\cdot\mathrm{diam}(K)$
for all $x\in[a,b]$ and $l\in\mathbb{N}$. So the sequence $\left(f_{k}\right)_{=1}^{\infty}$
is uniformly convergent and $f(x):=\lim_{k\rightarrow\infty}f_{k}(x)$
is a continuous function. Then $graph(f)$ is a non-degenerate continuous
curve, that is contained in $K$. Hence by Lemma \ref{lem:gorbe resz lem}
there exists a $1$-dimensional $C^{1}$ submanifold $\Gamma$ of
$\mathbb{R}^{2}$ such that $\mathcal{H}^{1}(K\cap\Gamma)>0$. Thus
$K$ is contained in a line by Corollary \ref{cor: manifold intersection cor}
but this contradicts with the assumption of the theorem. So there
exists $\boldsymbol{\mathbf{i}}\in\mathcal{I}^{k}$ for some $k\in\mathbb{N}$
such that $K_{\boldsymbol{\mathbf{i}}}$ is not a fitting piece.

Let $K_{\boldsymbol{\mathbf{i}}}$ be a non-fitting piece for some
$\boldsymbol{\mathbf{i}}\in\mathcal{I}^{k}$. Then there exists $x\in K$
such that $x\notin K_{\boldsymbol{\mathbf{i}}}$ and $P(x)\in P_{0}(K_{\boldsymbol{\mathbf{i}}})$.
Because $K_{\boldsymbol{\mathbf{i}}}$ is compact and $x\notin K_{\boldsymbol{\mathbf{i}}}$
it follows that $\mathrm{dist}(x,K_{\boldsymbol{\mathbf{i}}})>0$.
Let $k_{2}\in\mathbb{N}$ such that $r_{max}^{k_{2}}\cdot\mathrm{diam}(K)<\mathrm{dist}(x,K_{\boldsymbol{\mathbf{i}}})$.
There exists $\boldsymbol{\mathbf{j}}\in\mathcal{I}^{k_{2}}$ such
that $x\in K_{\boldsymbol{\mathbf{j}}}$. It follows that $K_{\boldsymbol{\mathbf{i}}}\cap K_{\boldsymbol{\mathbf{j}}}=\emptyset$,
$P_{0}(K_{\boldsymbol{\mathbf{i}}})\cap P_{0}(K_{\boldsymbol{\mathbf{j}}})\neq\emptyset$
but by Lemma \ref{lem: SS Almost disjoint proj} $\left|P_{0}(K_{\boldsymbol{\mathbf{i}}})\cap P_{0}(K_{\boldsymbol{\mathbf{j}}})\right|=1$.
Let $A=K_{\boldsymbol{\mathbf{i}}}$ and $B=K_{\boldsymbol{\mathbf{j}}}$.
By checking, similarly as we did it in the proof of Theorem \ref{thm:SS int-proj SSC},
that all the conditions of Lemma \ref{lem: Key lemma} are satisfied
for $A,B\subseteq K$ it follows that there exists $\delta>0$ such
that $P_{0+\varphi}\notin IP(K)$ for every $\varphi\in(-\delta,\delta)\setminus\{0\}$.$\hfill\square$

\begin{rem}
Let $K\subseteq\mathbb{R}^{d}$. We can define interval projections
in higher dimension such that we call \textit{$\Pi_{M}$ an interval
projection of $K$} for a line $M\subseteq\mathbb{R}^{d}$ through
the origin, if $\Pi_{M}(K)$ is an interval with the $M=\mathbb{R}$
identification. For a set $K\subseteq\mathbb{R}^{d}$ we denote the
set of all interval projections of $K$ by $IP(K)$. We can generalize
Theorem \ref{thm:SS int-proj OSC} as follows:
\end{rem}

\textit{Let $\left\{ S_{i}\right\} _{i=1}^{m}$ be an SS-IFS in $\mathbb{R}^{d}$
($d\geq2$) with attractor $K$ such that each $S_{i}$ are homotheties
and $K$ is a $1$-set that is not contained in any affine hyperplane.
Then there are at most finitely many lines through the origin, such
that the orthogonal projection onto them are interval projections
of $K$.}

The assumption, that $K$ is not contained in any affine hyperplane,
is necessary. Let $K$ be a set contained in a hyperplane $H$, let
$v$ be the normal vector of $H$, let $\Pi_{M}$ be an interval projection
of $K$ for some line $M\subseteq H$, and let $x$ be a non-zero
vector in $M$. Let $M_{n}=\left\{ \lambda\cdot(x+\frac{1}{n}\cdot v):\lambda\in\mathbb{R}\right\} $.
Then $\Pi_{M_{n}}$ is an interval projection of $K$, hence $\left|IP(K)\right|=\infty$.

For an $l$-dimensional subspace $M\subseteq\mathbb{R}^{d}$ we denote
its orthogonal direct complement by $M^{\perp}$. For a vector $v\in\mathbb{R}^{d}$
let $v^{\perp}$ be the hyperplane through the origin that is orthogonal
to $v$. We denote the set of all lines in $\mathbb{R}^{d}$ through
the origin by $G_{d,1}$. \\

\noindent \textit{Sketch of the proof.} Similarly to the proof of
Lemma \ref{lem: Compact IP lem} one can show that $IP(K)$ is compact.
Thus it is enough to show that every element of $IP(K)$ is isolated.
We show it by induction on $d$. The statement holds for $d=2$ by
Theorem \ref{thm:SS int-proj OSC}. Let $d>2$ and $\Pi_{M}\in IP(K)$
for some line $M$. Just as in the proof of Theorem \ref{thm:SS int-proj OSC},
we can show that for an endpoint $a$ of the interval $\Pi_{M}(K)$
there exists a unique $w\in K$ such that $\Pi_{M}(w)=a$. Moreover,
one can show that there exists a closed convex cone $C\subseteq\mathbb{R}^{d}$
such that $w\in C$ is the apex, $K\subseteq C$ and $a\notin\Pi_{M}(C\setminus w)$.
With a similar argument to that in the proof of Theorem \ref{thm:SS int-proj OSC}
about fitting pieces, one can show that there exist $\boldsymbol{\mathbf{i}},\boldsymbol{\mathbf{j}}\in\bigcup_{k=1}^{\infty}\mathcal{I}^{k}$
such that $K_{\boldsymbol{\mathbf{i}}}\cap K_{\boldsymbol{\mathbf{j}}}=\emptyset$
and $\left|\Pi_{M}(K_{\boldsymbol{\mathbf{i}}})\cap\Pi_{M}(K_{\boldsymbol{\mathbf{j}}})\right|=1$.
Let $u\in K_{\boldsymbol{\mathbf{i}}}$ and $v\in K_{\boldsymbol{\mathbf{j}}}$
such that $\Pi_{M}(u)=\Pi_{M}(v)$. Let $N$ be the $2$-dimensional
linear subspace that contains $M$ and the vector $u-v$ and identify
$N$ with $\mathbb{R}^{2}$ such that $M$ is the $x$-axis and $u-v$
is parallel to the $y$-axis. Using similar ideas to those that were
used to prove Lemma \ref{lem: Key lemma} one can show that there
exists a neighbourhood $U$ of $M$ in $G_{d,1}$ such that if $M_{2}\in U$
and $N\cap M_{2}^{\perp}$ is not parallel to $u-v$ then $\Pi_{M_{2}}$
is not an interval projection of $K$ (we note that in this argument
we use the fact that $K\subseteq C$ and $a\notin\Pi_{M}(C\setminus w)$).
To finish the proof we need to show that there are only finitely many
$M_{2}\in IP(K)$ such that $N\cap M_{2}^{\perp}$ is parallel to
$u-v$. Notice that if $N\cap M_{2}^{\perp}$ is parallel to $u-v$
then $M_{2}\subseteq(u-v)^{\perp}$ and so $\Pi_{M_{2}}=\Pi_{M_{2}}\circ\Pi_{(u-v)^{\perp}}$.
Hence if $\Pi_{M_{2}}\in IP(K)$ then $\Pi_{M_{2}}\in IP(\Pi_{(u-v)^{\perp}}(K))$
in $(u-v)^{\perp}=\mathbb{R}^{d-1}$. However, by the inductive assumption
$IP(\Pi_{(u-v)^{\perp}}(K))$ is a finite set because $\Pi_{(u-v)^{\perp}}(K)$
is a self-similar set since each $S_{i}$ is a homothety.$\hfill\square$

\section{Examples\label{sec:Examples}}

For an SS-IFS $\left\{ S_{i}\right\} _{i=1}^{m}$ with attractor $K$
we call the unique solution $s$ of the equation
\begin{equation}
\sum_{i=1}^{m}r_{i}^{s}=1\label{eq: S-dim sum}
\end{equation}
the \textit{similarity dimension} of the SS-IFS. A straightforward
covering argument shows that $\dim_{H}K\leq s$ and $\mathcal{H}^{s}(K)<\infty$,
see for example \cite[5.1 Prop(4)]{Hutchinson}. Let $0<r<1$ be such
that $r_{i}\leq r$ for every $i\in\mathcal{I}$. Then $1=\sum_{i=1}^{m}r_{i}^{s}\leq m\cdot r^{s}$
and hence
\[
\dim_{H}K\leq s\leq\log(m)/\log(1/r).
\]
In this section the dimension estimation of the self-similar sets
will all be based on this formula, hence we will use it without any
reference.

If there exists a compact, convex set $F$ of nonempty interior such
that $S_{i}(F)\subseteq F$ for every $i\in\mathcal{I}$ and we define
$F_{k}:=\bigcup_{\boldsymbol{\mathbf{i}}\in\mathcal{I}^{k}}S_{\boldsymbol{\mathbf{i}}}(F)$
then $F\supseteq F_{1}\supseteq F_{2}\supseteq\dots$ and since $K$
is the unique compact attractor of the SS-IFS it follows that $K=\bigcap_{k=1}^{\infty}F_{k}$.
Assume furthermore, that $S_{i}$ is a homothety for every $i\in\mathcal{I}$
and for some $\theta\in\mathbb{R}$ every line $L$ that is parallel
to $L_{\frac{\pi}{2}+\theta}$ and $L\cap F\neq\emptyset$ we have
that $L\cap F_{1}\neq\emptyset$. Then $L\cap F_{k}\neq\emptyset$
for every $k\in\mathbb{N}$ and so $L\cap K\neq\emptyset$. Hence
$P_{\theta}(K)=P_{\theta}(F)$ is an interval. It also implies that
$\dim_{H}K\geq\dim_{H}P_{\theta}(K)=1$ and $\mathcal{H}^{1}(K)\geq\mathcal{H}^{1}(P_{\theta}(K))>0$
because orthogonal projection does not increase the Hausdorff dimension
and measure. If for some $\theta\in\mathbb{R}$ we can find such $F$
then we call $F$ a \textit{$\theta$-witness of interval projection}
for the SS-IFS or shortly just say $F$ is a \textit{witness for $M=L_{\theta}$}.
\begin{example}
\label{exa:1dim-sierp}There exists a self-similar $1$-set with three
interval projections such that the projection intervals have the same
length. The \textit{$1$-dimensional Sierpinski triangle} is the attractor
of the SS-IFS that contains three homotheties which map an equilateral
triangle into itself fixing the corners with similarity ratio $1/3$.
If $M$ is a line that contains a side $[a,b]$ of the equilateral
triangle then $\Pi_{M}(K)=[a,b]$ because the equilateral triangle
is a witness for $M$. Hence $IP(K)$ contains at least three projections.
It is easy to show that $IP(K)$ contains exactly three projections.
See Figure 5.

\includegraphics[scale=0.33]{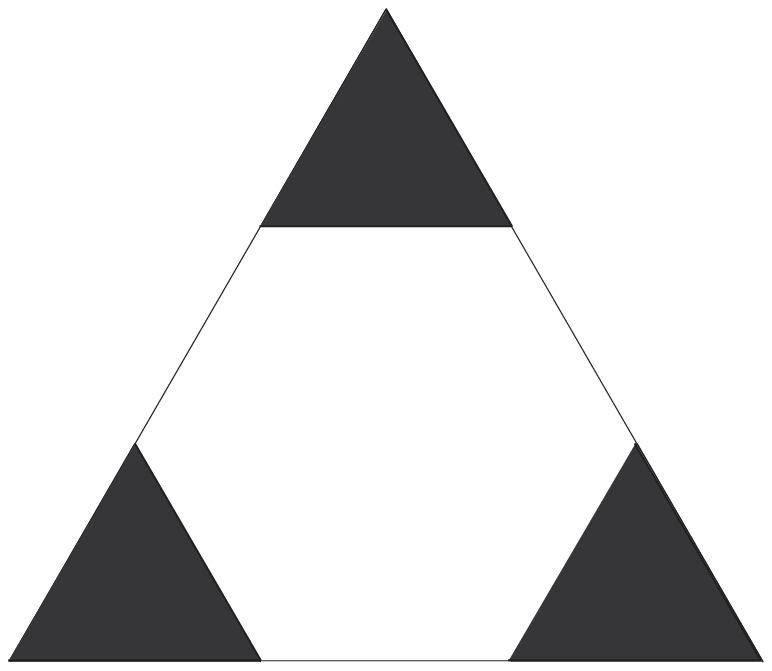}~~~~~~~~~~\includegraphics[scale=0.33]{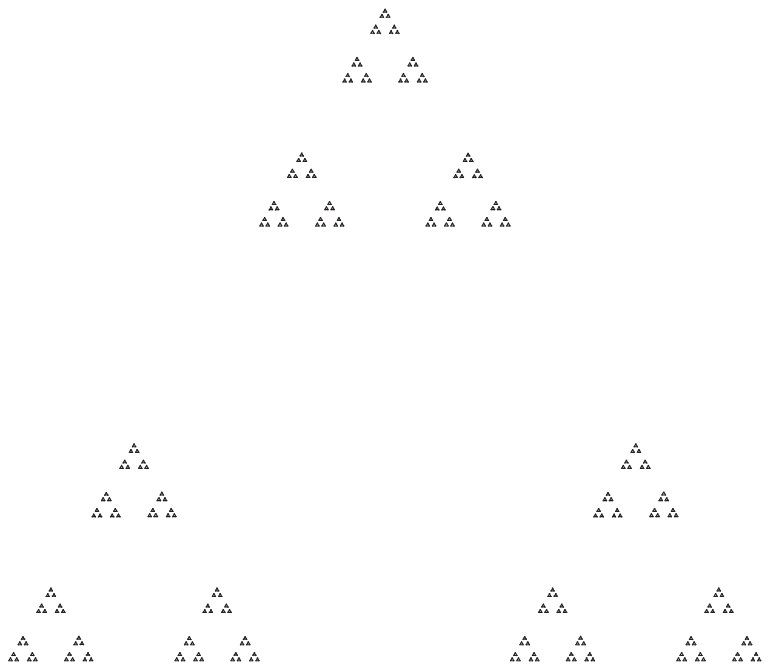}~~~~~~~~~~~~~~\includegraphics[scale=0.23]{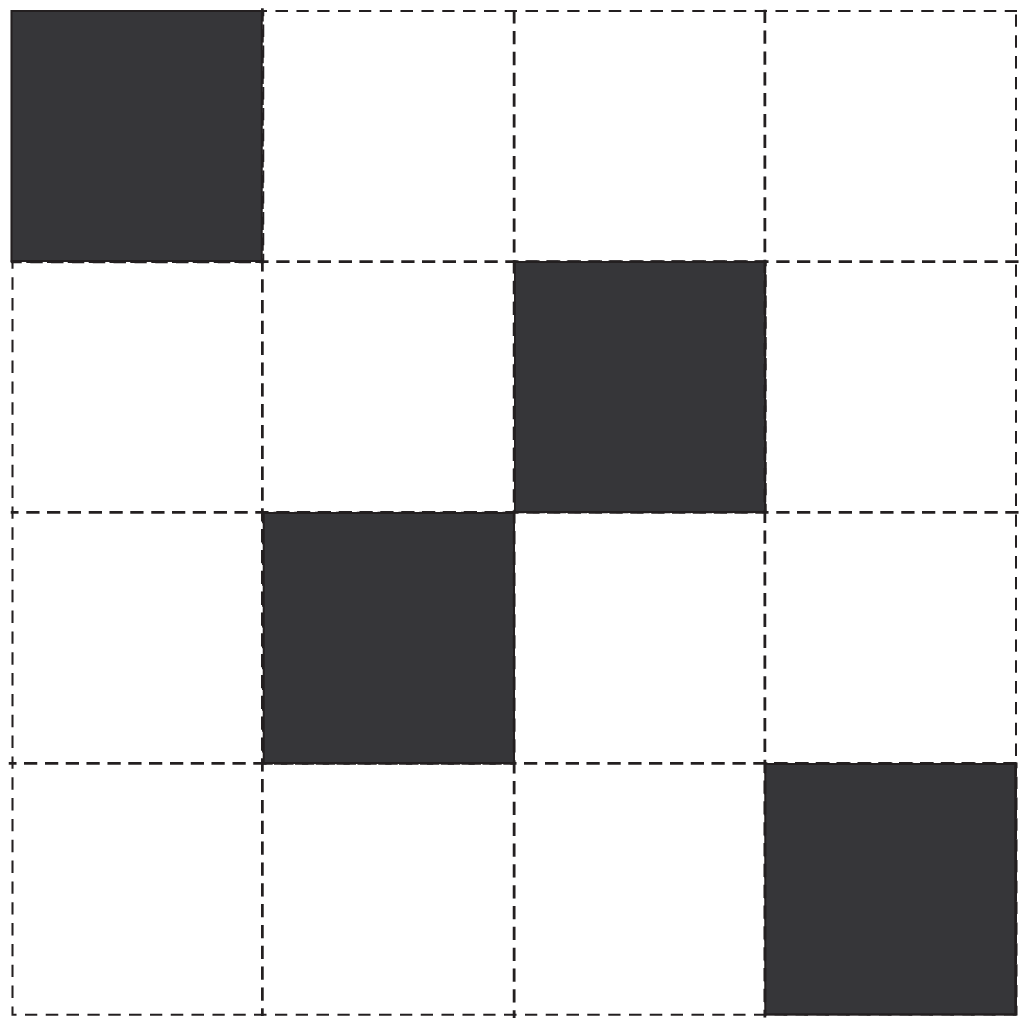}~~~~~~~~~~~~~~~\includegraphics[scale=0.23]{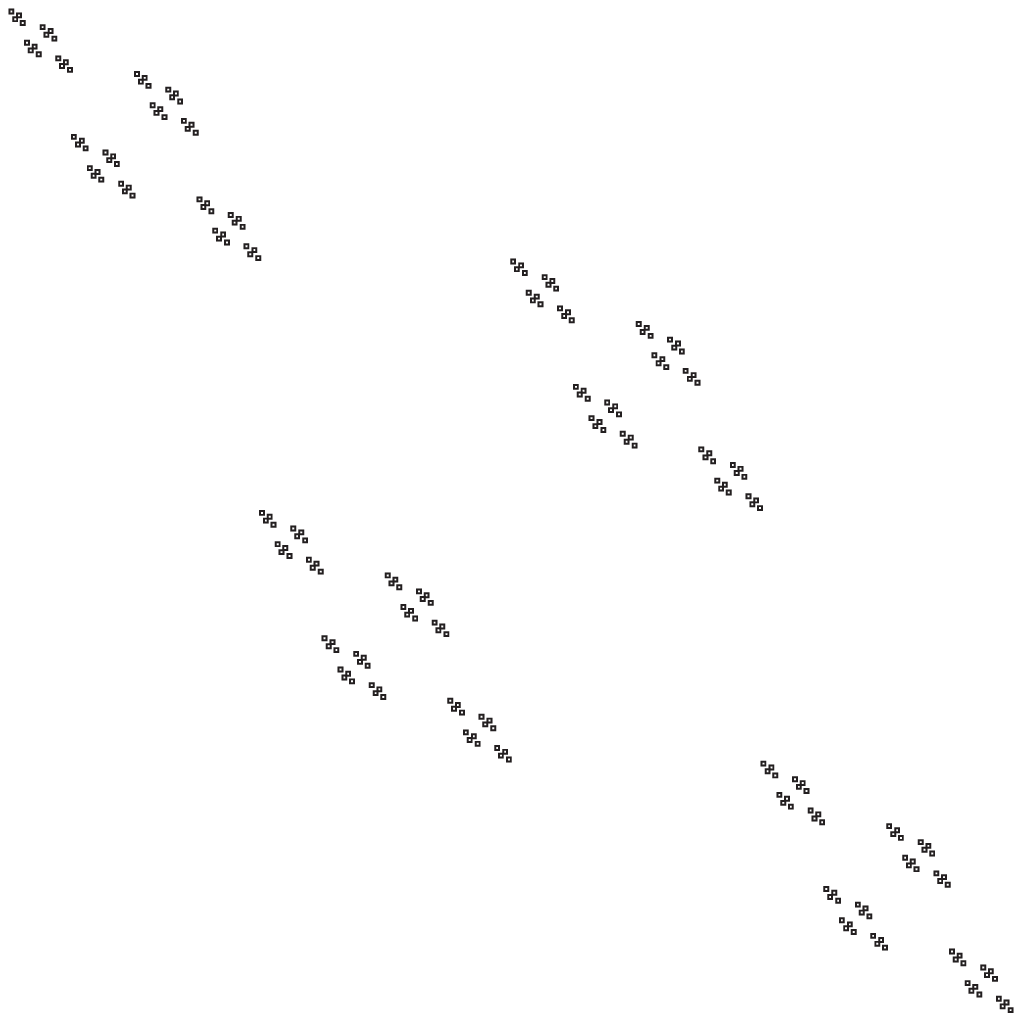}

~~~~~~~~~~~~~~~~~~~~~~Figure 5.~~~~~~~~~~~~~~~~~~~~~~~~~~~~~~~~~~~~~~~~~~~~~~~~~~~~~~~~~~~~~Figure
6.
\end{example}

\begin{example}
\label{exa:negyzet-ip}There exists a self-similar $1$-set that has
four interval projections. We take four homotheties of similarity
ratio $1/4$ that map the unit square into itself as it is shown on
Figure 6. The projection of the attractor on both the $x$- and $y$-axes
are intervals of length $1$ because the unit square is witness for
those. Let $F$ be the rhombus that is the convex hull of the fixed
points of the homotheties. Then the homotheties map $F$ into itself
fixing the corners. Again $F$ is a witness for the coordinate axes.
However, $F$ is a witness for two further lines, see Figure 7 and
Figure 8. Checking on $F_{1}$ it is not hard to show that $IP(K)$
consist of exactly four lines.

~~~~~~~~~~~~~~\includegraphics[scale=0.33]{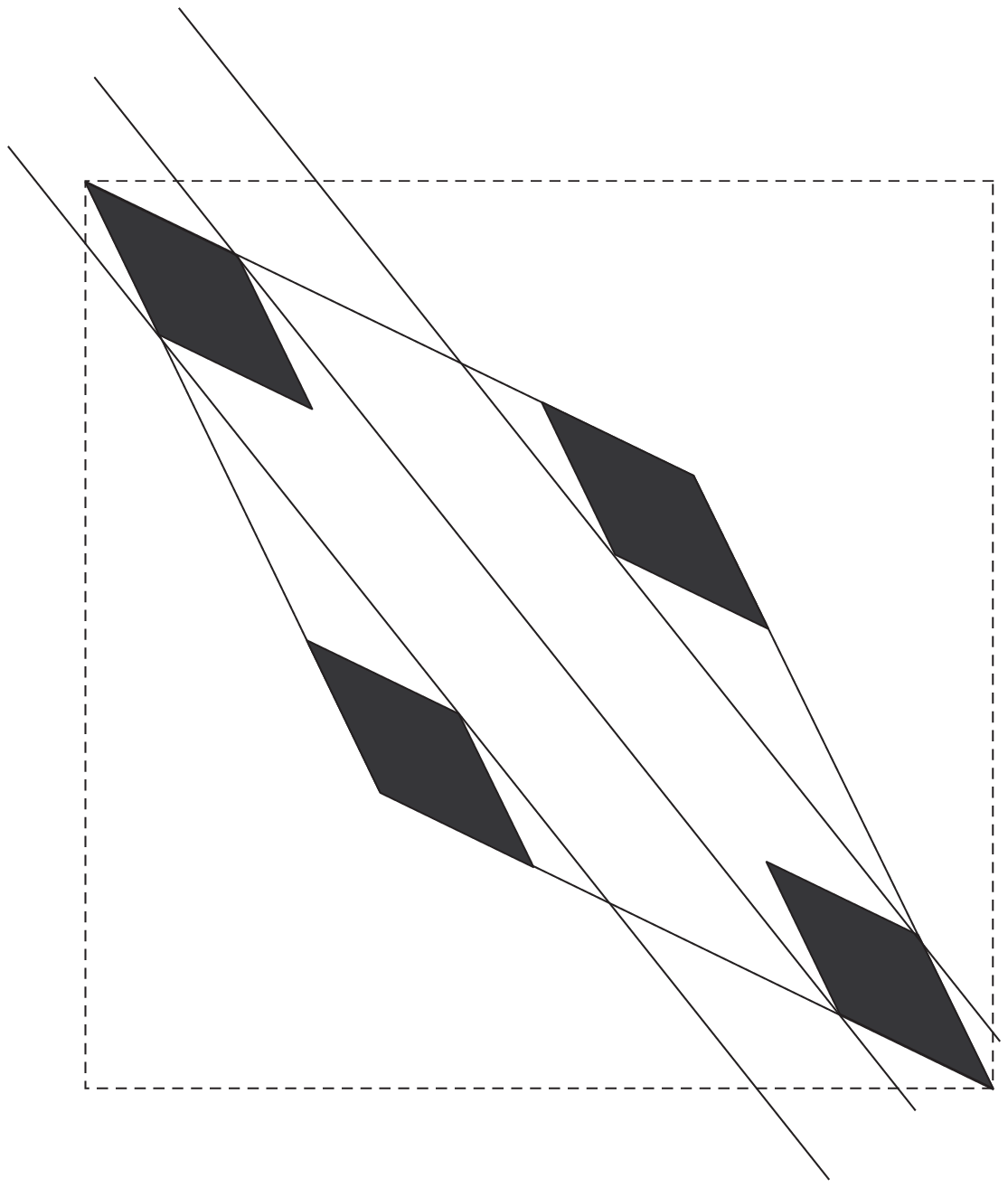}~~~~~~~~~~~~~~~~~~~~~~~~~~\includegraphics[scale=0.33]{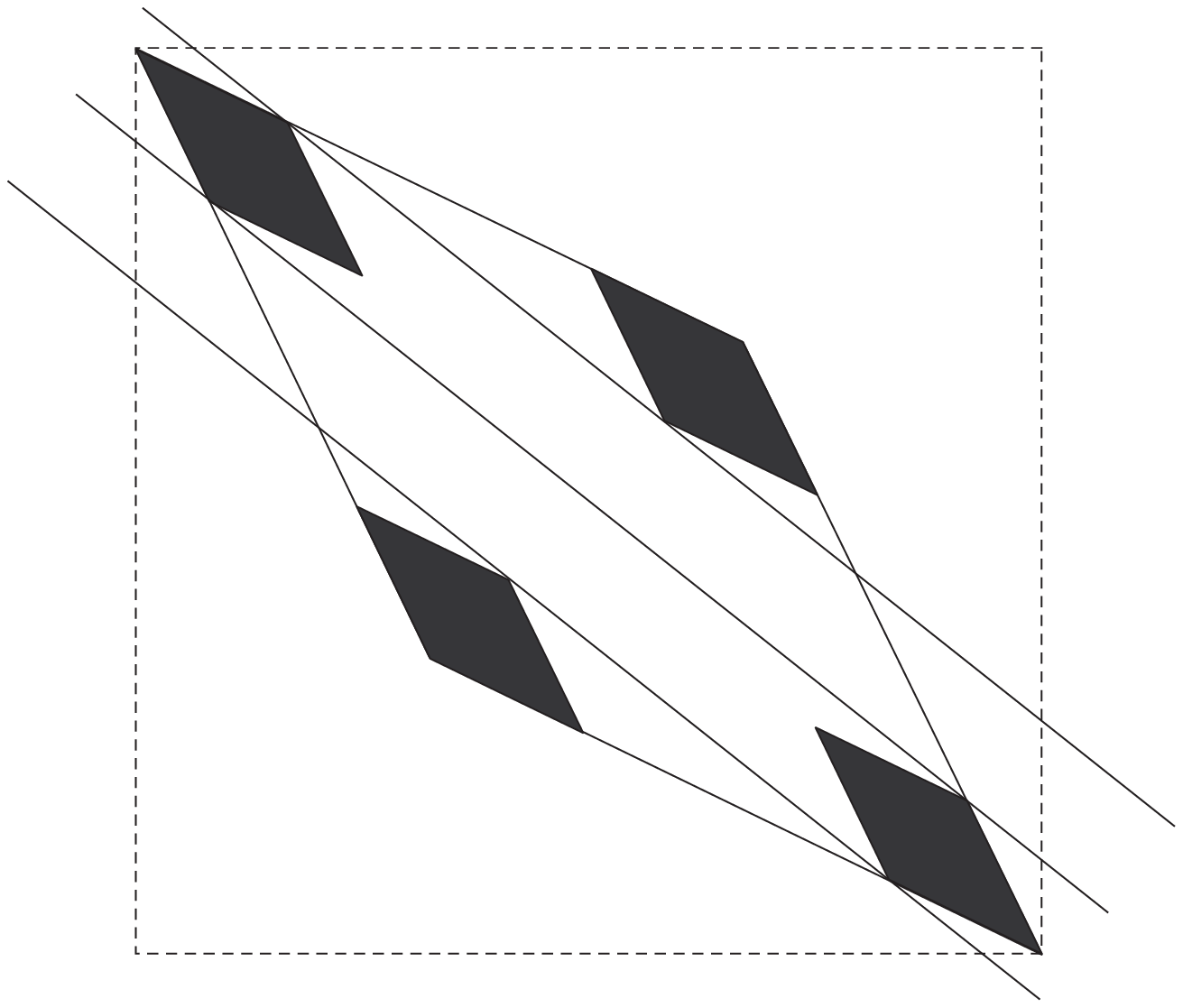}

~~~~~~~~~~~~~~~~~~~~~~~~~~Figure 7.~~~~~~~~~~~~~~~~~~~~~~~~~~~~~~~~~~~~~~~~~~~~~~~~~~~Figure
8.
\end{example}

\begin{rem}
In example \ref{exa:negyzet-ip} not all the four projection intervals
have the same length. One can easily verify it from that $F$ is a
witness for the projections. However, it is also easy to check that
$\int x\cdot y\mathrm{d}\mu(x,y)<0$ when $p$ is the centre of mass
of $\mathcal{H}^{1}\vert_{K}$ and $\mu=\mathcal{H}^{1}\vert_{K-p}$.
Hence the intervals could not be of same length by Theorem \ref{thm:SS moment of inertia}.
\end{rem}

\begin{example}
\label{exa:four corner}There exists a self-similar $1$-set with
four interval projections such that the projection intervals have
the same length. The \textit{$1$-dimensional four corner set} is
the attractor of the SS-IFS that contains four homotheties which map
the unit square into itself fixing the corners with similarity ratio
$1/4$. We indicate on Figure 9 what the interval projections are.
One can think of Example \ref{exa:negyzet-ip} as an affine image
of the $1$-dimensional four corner set.
\end{example}

\includegraphics[scale=0.246]{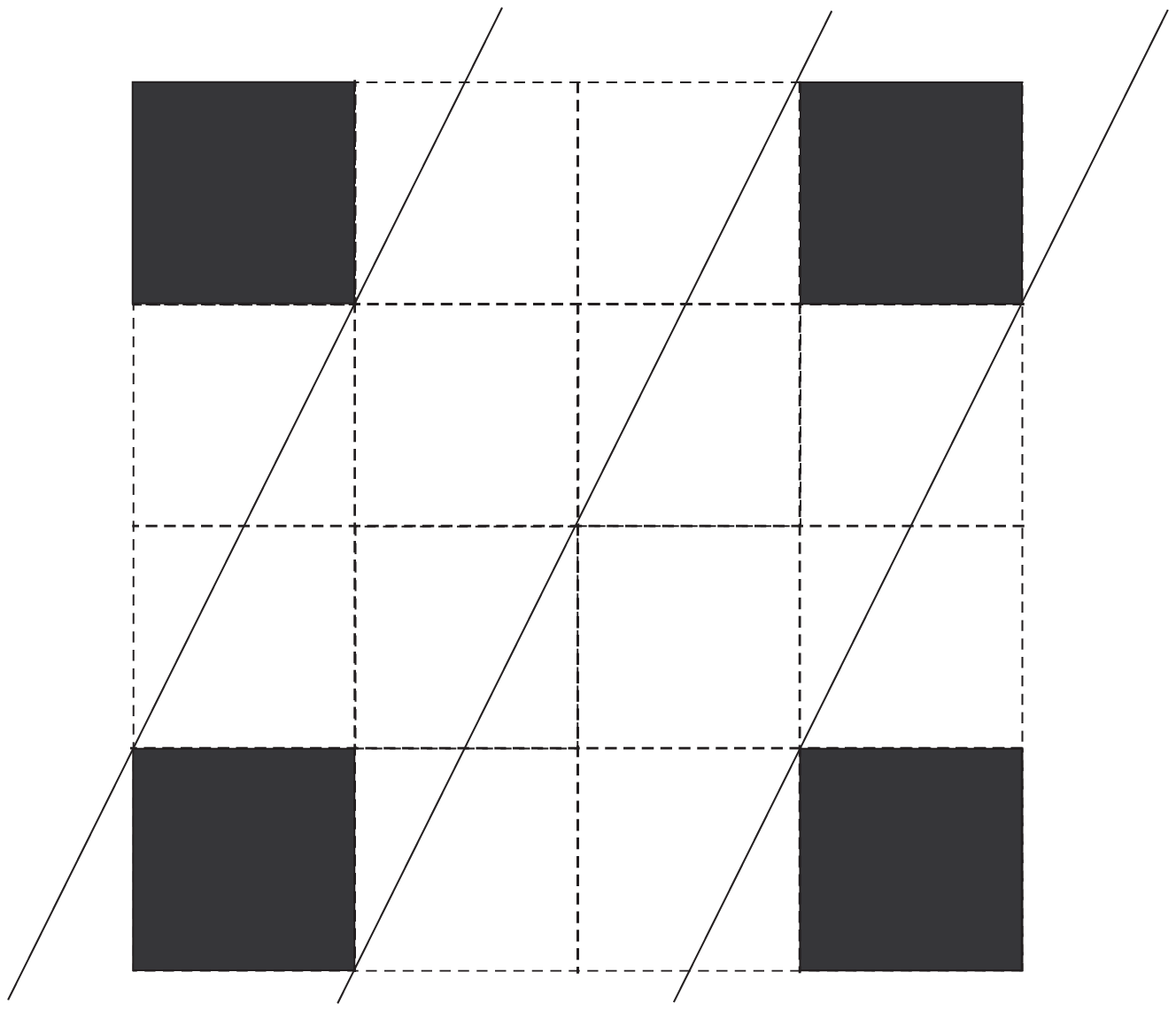}~~~~~~~\includegraphics[scale=0.246]{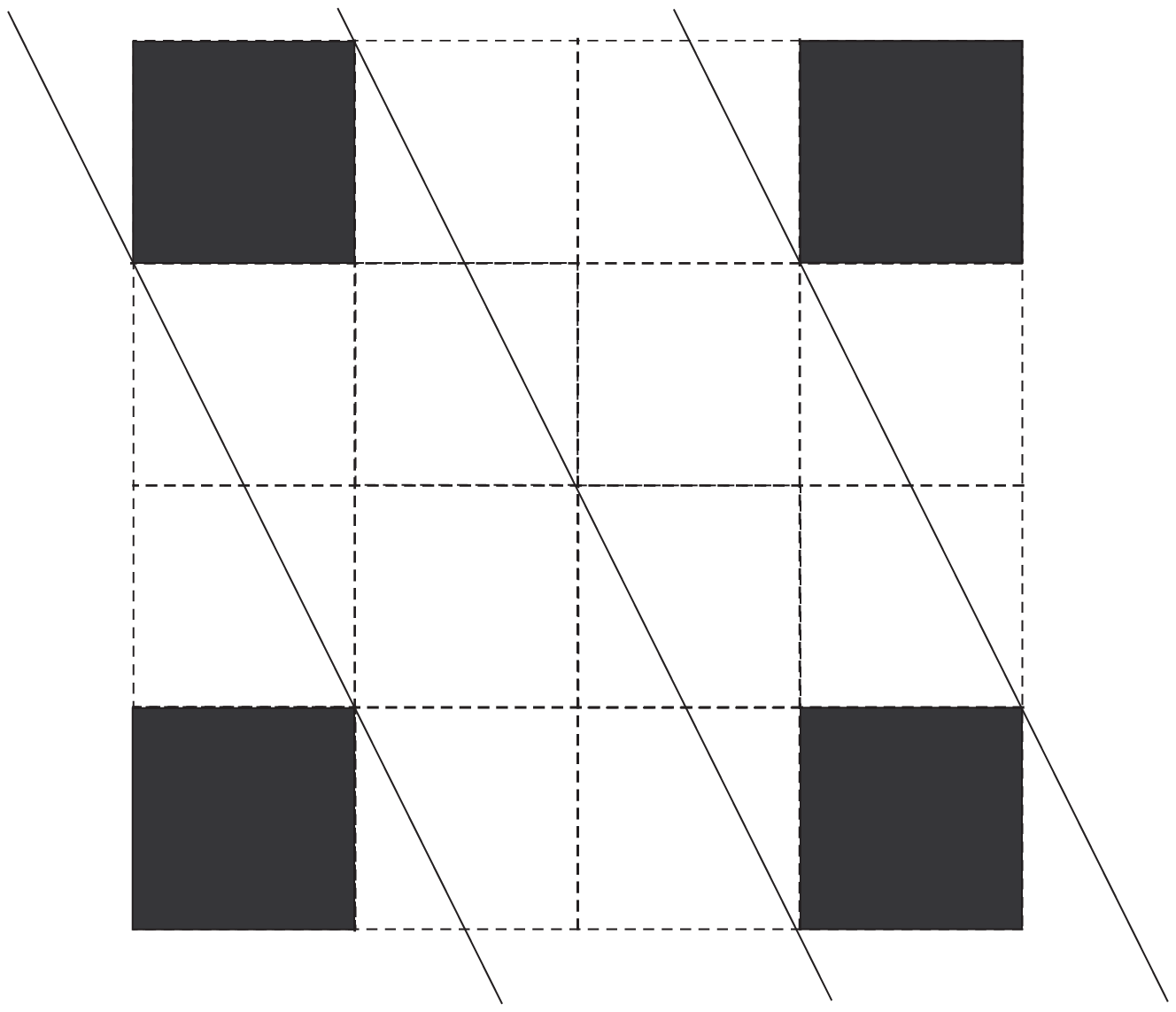}~~~~~~\includegraphics[scale=0.246]{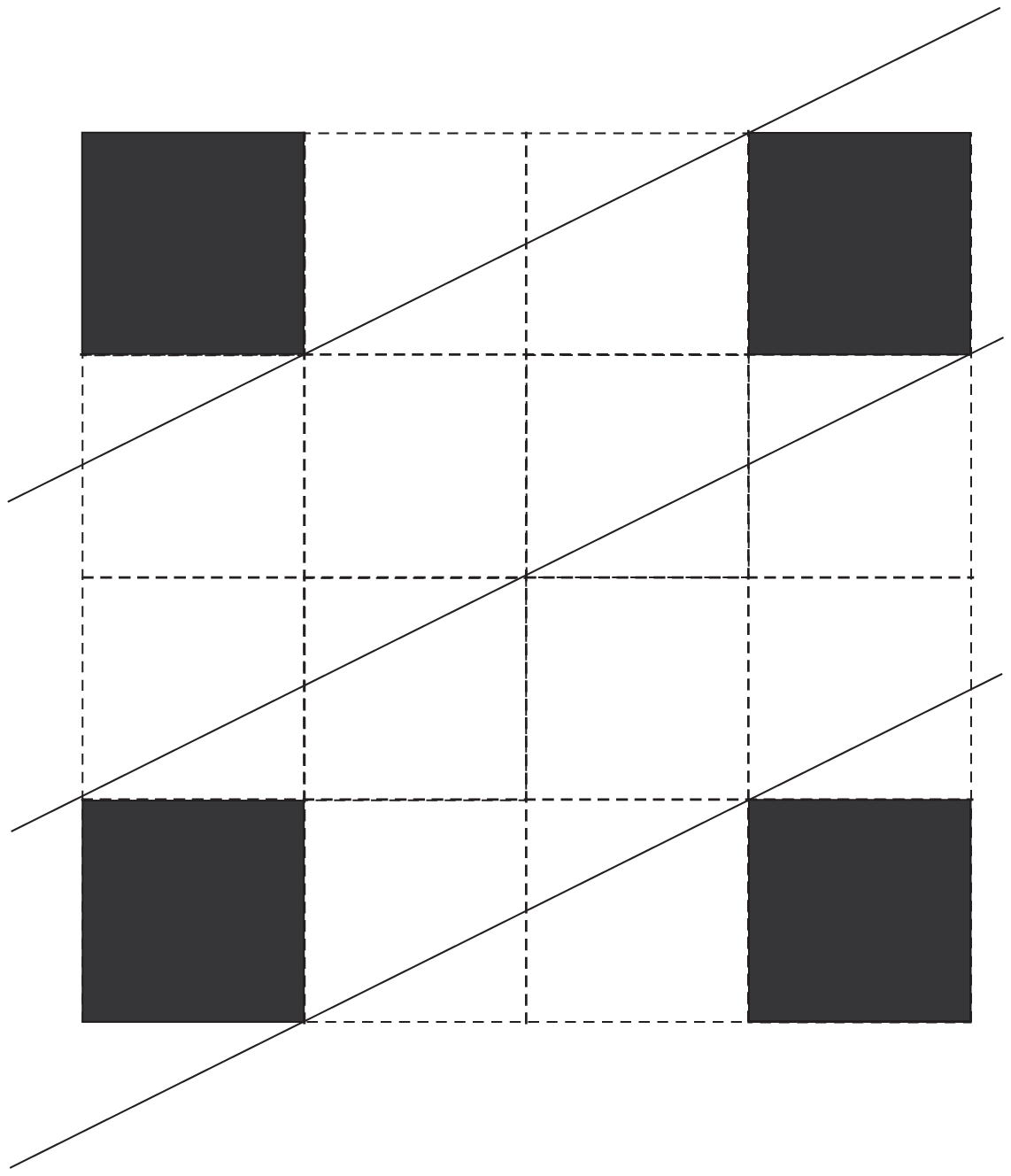}~~~~~~\includegraphics[scale=0.246]{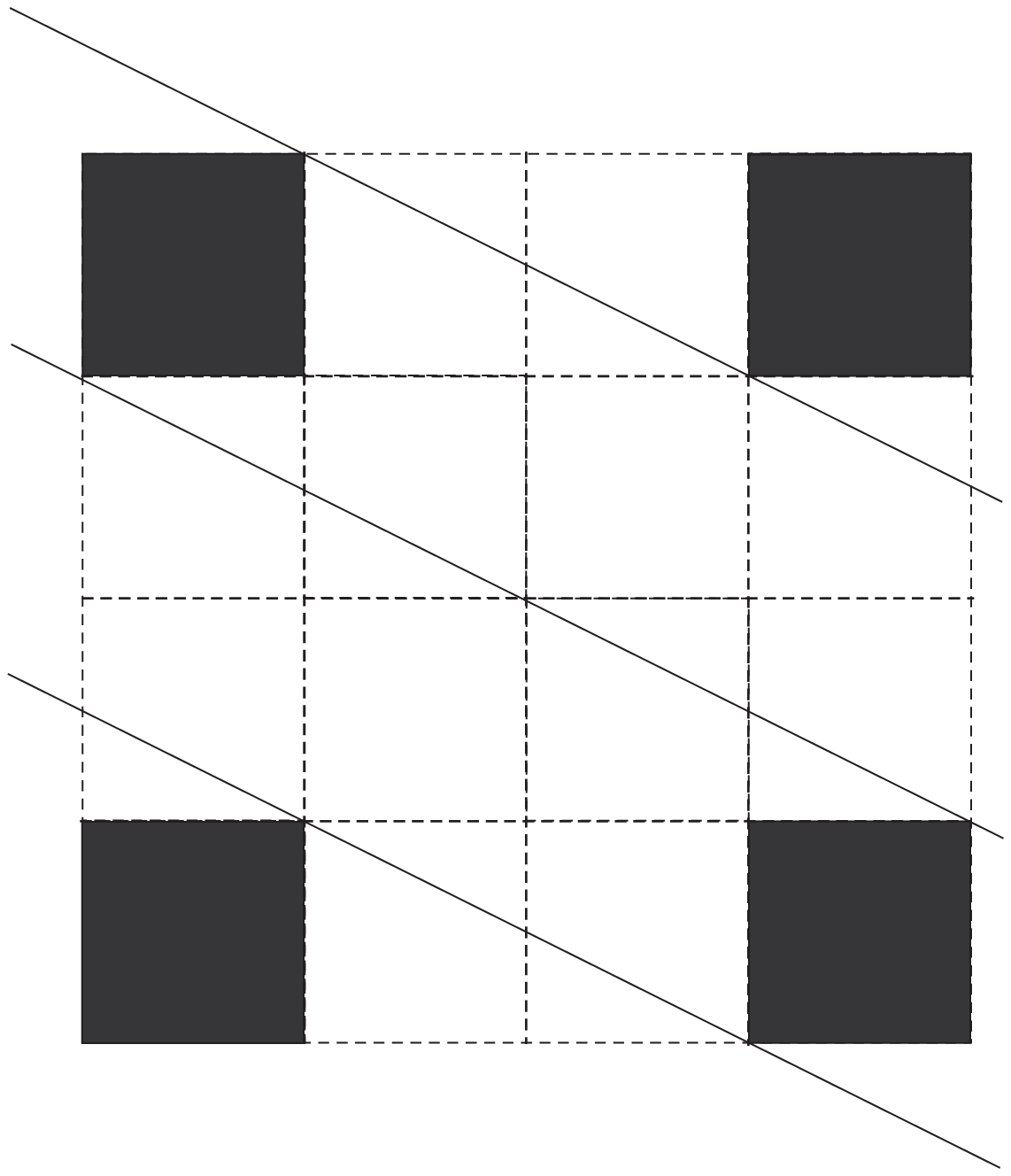}

~~~~~~~~~~~~~~~~~~~~~~~~~~~~~~~~~~~~~~~~~~~~~~~~~~~~~~~~~~~~~~Figure
9.

In the definition of a witness for a line $M=L_{\theta}$ we required
the similarities to be homtheties. If the similarities are not all
homotheties then we need to understand the intersection of $F$ and
$F_{1}$ not only with lines that are parallel to $L_{\pi/2+\theta}$
but also for lines that are parallel to $T(L_{\pi/2+\theta})$ for
every $T\in\mathcal{T}$. If we want to show that every projection
is an interval then we can define the witness the following way. Let
$F$ be a convex compact set with non-empty interior such that $S_{i}(F)\subseteq F$
for every $i\in\mathcal{I}$. Then $F\supseteq F_{1}\supseteq F_{2}\supseteq\dots$
and the attractor $K=\cap_{k=1}^{\infty}F_{k}$. Assume that every
line $L$ that intersects $F$ also intersects $F_{1}$. It follows
that every line $L$ that intersects $F$ also intersects $F_{k}$
for every $k\in\mathbb{N}$ and so $L$ intersects $K$. Hence every
projection is an interval. If we can find such a set $F$ then we
say that \textit{$\ensuremath{F}$ is a witness for every line}. Note
that for this definition we did not assume that the similarities are
homotheties.
\begin{example}
\label{Ex: Minen-vet-int}For every $\varepsilon>0$ there exists
a totally disconnected self-similar set $K$ of Hausdorff dimension
less then $1+\varepsilon$ such that every projection is an interval
projection of $K$. Rather than giving a complicated explicit description
of the SS-IFS we suggest the maps by drawing pictures. First consider
the SS-IFS when the similarities map the unit square into itself by
five maps as shown on Figure 10. It satisfies the SSC and the unit
square is a witness for every line. Hence every projection is an interval.

To decrease the Hausdorff dimension dimension we take more maps of
smaller similarity ratio $r$. We consider a sequence of SS-IFS, again
satisfying the SSC, when the similarities map the unit square into
itself such that images are close to the diagonals as shown on Figure
11. The number of maps that the SS-IFS consist of is $O(1/r)$ as
$r$ goes to $0$. Thus the similarity dimension approaches $1$ as
$r$ goes to $0$. Again the unit square is a witness for every line.

In the previous approach the similarities are not homotheties. We
can also give examples where every similarity is a homothety. Consider
the sequence of SS-IFS, now only satisfying the `open set condition',
when the similarities map the unit square into itself such that the
images touch the sides of the unit square as shown on Figure 12. Again
the number of maps that the SS-IFS consist of is $O(1/r)$ and the
unit square is a witness for every line.

We note that one could construct an example when both the SSC is satisfied
and every similarity is a homothety. However, the constructions is
a little bit more complicated to present.
\end{example}

\includegraphics[scale=0.265]{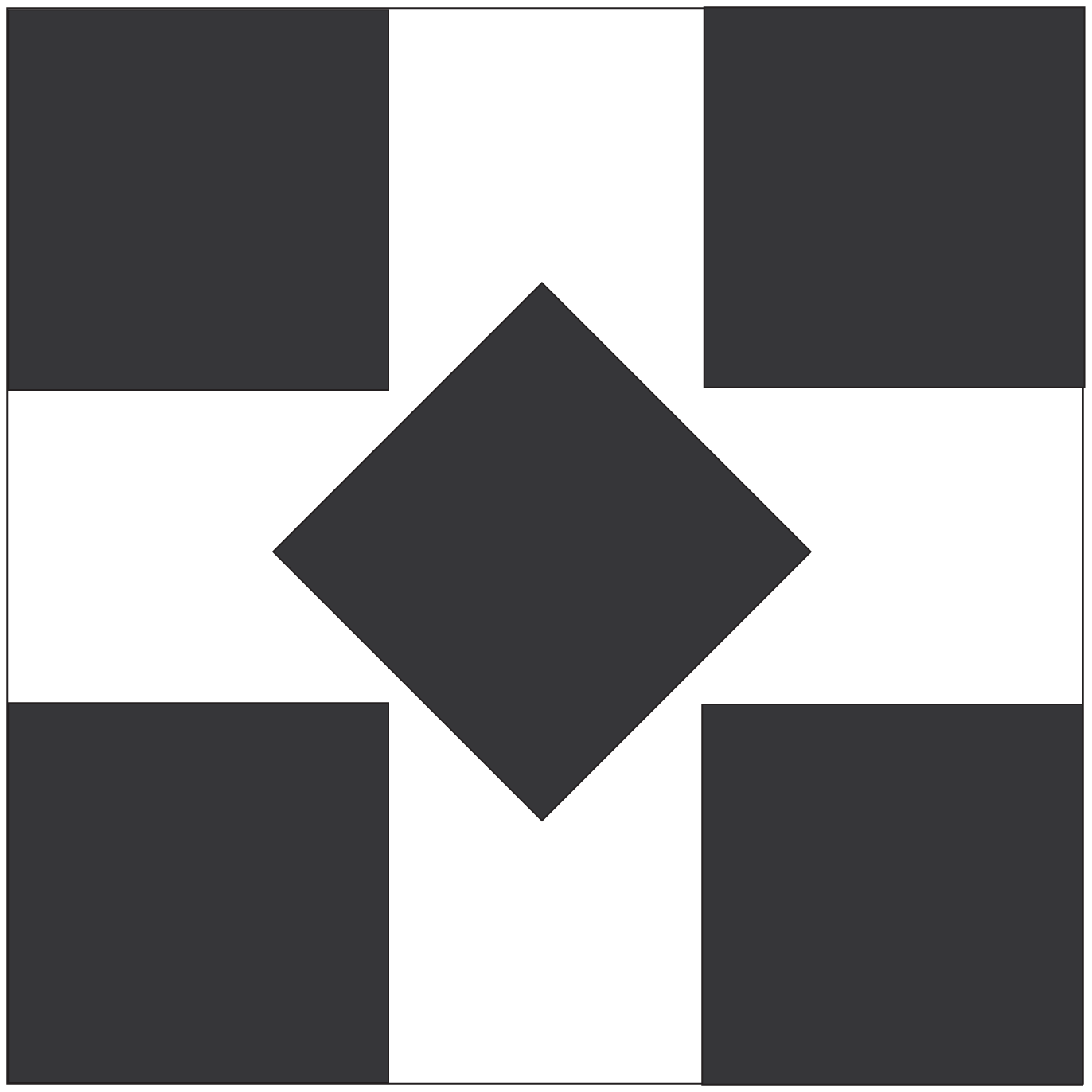}~~~~~~~~~~~~~~~~\includegraphics[scale=0.22]{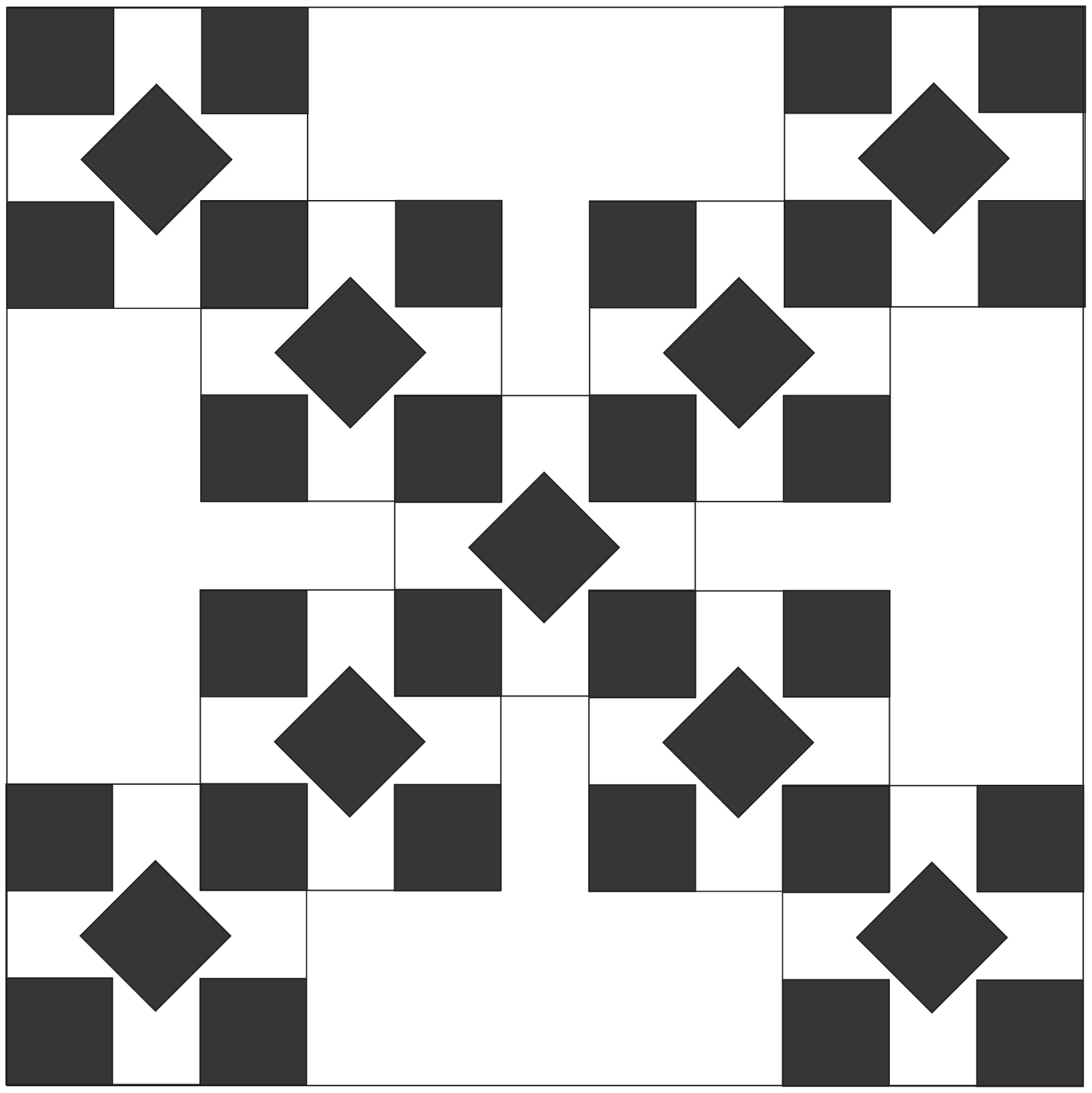}~~~~~~~~~~~~~~~~\includegraphics[scale=0.25]{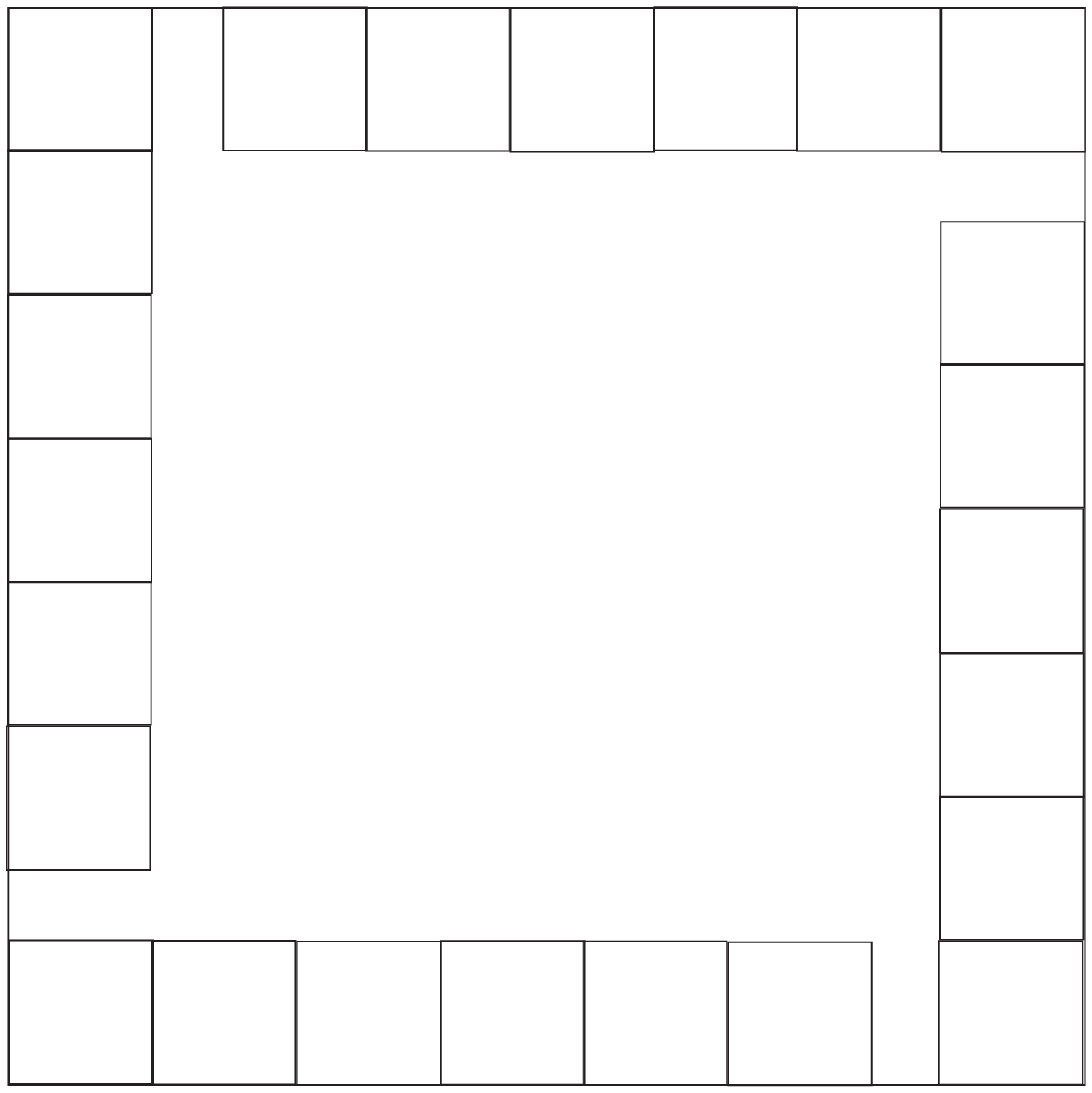}

~~~~~~~~~~Figure 10.~~~~~~~~~~~~~~~~~~~~~~~~~~~~~~~~~~~Figure
11.~~~~~~~~~~~~~~~~~~~~~~~~~~~~~~~~~~~Figure
12.
\begin{example}
\label{1dimfullip}There exists a totally disconnected, compact set
of Hausdorff dimension $1$ such that every projections is an interval
projection. For every $n\in\mathbb{N}$ let $\mathcal{I}_{n}$ be
an SS-IFS mentioned in Example \ref{Ex: Minen-vet-int} with the Hausdorff
dimension of the attractor being less than $1+1/n$. Let $F$ be the
unite square and
\[
F_{n}:=\bigcup\left\{ S_{i_{1}}\circ\ldots\circ S_{i_{n}}(F):i_{k}\in\mathcal{I}_{k},k=1,\dots,n\right\} .
\]
Then $F_{1}\supseteq F_{2}\supseteq\ldots$ and let $K=\bigcap_{n=1}^{\infty}F_{n}$
. By a standard covering argument it is easy to show that $\mathcal{H}^{1+\varepsilon}(K)<\infty$
for every $\varepsilon>0$ and hence $\dim_{H}K\leq1$. Via an argument
similar to the witness of projections one can show that every projection
is an interval projection.
\end{example}

\begin{example}
\label{exa:christmas chaain}There exists a totally disconnected self-similar
set $K$ in $\mathbb{R}^{3}$ such that the projection of $K$ onto
every $2$-dimensional plane is path connected. Note that if $M\subseteq\mathbb{R}^{3}$
is a $2$-dimensional plane, $L\subseteq M$ is a line and $\Pi_{M}(K)$
is connected then $\Pi_{L}(K)$ is an interval because $\Pi_{L}\circ\Pi_{M}=\Pi_{L}$.

In $\mathbb{R}^{3}=\left\{ (x,y,z):x,y,z\in\mathbb{R}\right\} $ let
$C=\left\{ (x,y,0):x^{2}+y^{2}=1\right\} $ be the unit circle and
for the sake of this example we say that the points $(1,0,0)$ and
$(-1,0,0)$ are the end points of $C$. Let $q>0$ be `very small'
(it will be clear from the construction how small) and let $F=\left\{ p\in\mathbb{R}^{3}:\mathrm{dist}(C,p)\leq q\right\} $
be the solid ring. Consider the SS-IFS containing similarities that
map $F$ into itself such that the image of the end points of $C$
are mapped into $C$ for every similarity and the union of the images,
that is $F_{1}$, form a chain inside $F$, like a necklace, where
the images of $F$ are the links. If $q$ is small enough we can do
this such that the images, i.e. the links, are disjoint. Hence the
strong separation condition is satisfied and so the attractor $K$
is totally disconnected. Fix a plane $M$. By taking a sufficient
uniformly convergent sequence $f_{n}$ of curves in $\Pi_{M}(F_{n})$
one can show that $\Pi_{M}(K)$ is path connected. As in Example \ref{Ex: Minen-vet-int}
we can take the similarity ratio $r$ to be very small but still have
at most $O(1/r)$ many similarity in the SS-IFS. Hence we can construct
$K$ to have Hausdorff dimension arbitrarily close to $1$.
\end{example}

\begin{center}
$\mathbf{Acknowledgements}$
\par\end{center}

The author was partially supported by an ERC grant (grant number 306494).
The author would like to thank Kenneth Falconer for pointing out the
connection between the integral of projections and the moment of inertia.
The results of the present paper emerged from the author`s master
thesis research. The author would like to thank M\'arton Elekes for
the help in writing the master thesis.

\end{document}